\documentstyle[11pt,amssymb,amsmath,amscd,amsbsy,amsfonts,amsthm]{article}

\input xy
\xyoption{all}

\newtheorem{thm}{Theorem}[section]
\newtheorem{lem}[thm]{Lemma}
\newtheorem{prop}[thm]{Proposition}
\newtheorem{cor}[thm]{Corollary}
\newtheorem{rem}[thm]{Remark}

\DeclareMathOperator{\Z}{\mathbb{Z}}

\title{\textsc{Subtle characteristic classes for $Spin$-torsors}}
\author{Fabio Tanania}
\date{}

\begin{document}
	
	\maketitle
	
		\begin{abstract}
		Extending \cite{SV}, we obtain a complete description of the motivic cohomology with $\Z/2$-coefficients of the Nisnevich classifying space of the spin group $Spin_n$ associated to the standard split quadratic form. This provides us with very simple relations among subtle Stiefel-Whitney classes in the motivic cohomology of \v{C}ech simplicial schemes associated to quadratic forms from $I^3$, which are closely related to $Spin_n$-torsors over the point. These relations come from the action of the motivic Steenrod algebra on the second subtle Stiefel-Whitney class. Moreover, exploiting the relation between $Spin_7$ and $G_2$, we describe completely the motivic cohomology ring of the Nisnevich classifying space of $G_2$. The result in topology was obtained by Quillen in \cite{Q}.
	\end{abstract}

	\section{Introduction}
	
	Our main purpose in this work consists in an attempt of better understanding $Spin$-torsors, which are closely related to quadratic forms from $I^3$. These are extremely interesting and fascinating objects and, although they arise quite naturally in many areas of mathematics, there are still many open questions about them due to their complexity and richness. In this paper, we try to study $Spin$-torsors from a motivic homotopic point of view by using classifying spaces and characteristic classes in motivic cohomology. At first, we need to mention that in the motivic homotopic environment there are two types of classifying spaces, the \textit{Nisnevich} and the \textit{\'etale}. The difference between the two is particularly visible when one works with non special algebraic groups. Indeed, in this case, the two types of classifying spaces above mentioned have in general different cohomology rings and, therefore, different characteristic classes. From \cite{MV}, we know that torsors are classified by \'etale classifying spaces, nevertheless studying Nisnevich classifying spaces has shown to provide some advantages in the project of investigating them.
	
	Actually, an essential inspiration for our work lies in \cite{SV}, where the authors study torsors by using Nisnevich classifying spaces. They are mainly interested in $BO_n$, the Nisnevich classifying space of the orthogonal group associated to the standard split quadratic form $q_n$, which provides a key tool to study $O_n$-torsors over the point which are nothing else but quadratic forms. In particular, they compute the motivic cohomology ring with $\Z/2$-coefficients of $BO_n$. This happens to be a polynomial algebra over the motivic cohomology of the point generated by some cohomology classes which are called \textit{subtle Stiefel-Whitney classes}. These are very informative invariants, for example they enable to recognise the power of the fundamental ideal of the Witt ring where a quadratic form belongs and they are also connected to the $J$-invariant introduced in \cite{V}. In a completely analogous way, it is possible to compute the motivic cohomology of $BSO_n$, which again is a polynomial algebra generated by all the subtle Stiefel-Whitney classes but the first, as one would expect from the classical topological result.
	
	In this work we go a bit further on this path by providing a complete description of the motivic cohomology with $\Z/2$-coefficients of $BSpin_n$, the Nisnevich classifying space of the spin group associated to the standard split form $q_n$. As we have already mentioned, this is a step forward in the understanding of $Spin$-torsors, and so of quadratic forms with trivial discriminant and Clifford invariant. In topology the singular cohomology of $BSpin_n$ was computed by Quillen in \cite{Q}. Essentially, his computation is based on two key tools: 1) the regularity of a certain sequence in the cohomology ring of $BSO_n$; 2) the Serre spectral sequence associated to the fibration $BSpin_n \rightarrow BSO_n$. Regarding 1), we essentially prove the regularity of a sequence in the motivic setting similar to Quillen's sequence in topology. This sequence is obtained from the second subtle Stiefel-Whitney class by acting with some specific Steenrod operations. As we will notice, the motivic situation is much more complicated than the topological one. This comes from the fact that in the motivic picture the element $\tau$ appears. Regarding 2), we use instead techniques developed in \cite{SV} to deal with fibrations of simplicial schemes with fibers which are motivically Tate, since in the motivic setting we lack a spectral sequence of Serre's type associated to a fibration. As a result, we get a description of the entire cohomology ring of $BSpin_n$ which is similar to the topological one in the same way as it is for the orthogonal and the special orthogonal cases. More precisely, we prove the following theorem (see Theorem \ref{MQ2}).
	
	\begin{thm}
		For any $n \geq 2$, there exists a cohomology class $v_{2^{k(n)}}$ of bidegree $(2^{k(n)-1})[2^{k(n)}]$ such that the natural homomorphism of $H$-algebras 
		$$H(BSO_n)/I_{k(n)} \otimes_H H[v_{2^{k(n)}}] \rightarrow H(BSpin_n)$$
		is an isomorphism, where $I_{k(n)}$ is the ideal generated by $\theta_0, \dots ,\theta_{k(n)-1}$ and $k(n)$ depends on $n$ as in the table of Theorem \ref{Q1}.
	\end{thm}

Equivalently, one can visualize $I_{k(n)}$ as the ideal generated by the action of the motivic Steenrod algebra on the second subtle Stiefel-Whitney class. This way we obtain subtle classes for $Spin$-torsors and relations among them. Moreover, by exploiting the relation between $Spin_7$ and the exceptional group $G_2$, we prove the following result that completely describes the motivic cohomology of $BG_2$ providing subtle characteristic classes for $G_2$-torsors, namely octonion algebras (see Theorem \ref{g2}).

\begin{thm}
	The motivic cohomology ring of $BG_2$ is completely described by
	$$H(BG_2) \cong H[u_4,u_6,u_7].$$
\end{thm}

Since torsors are classified by \'etale classifying spaces, much attention has been devoted to investigate their Chow rings (see \cite{To}), which neverthless are notoriously difficult to study. Regarding $Spin_n$, the picture is completely understood for $n \leq 6$ where the spin groups are known to be special by the sporadic isomorphisms. Guillot computed the Chow ring of the first non-trivial case, namely $Spin_7$, together with the one of $G_2$, over complex numbers in \cite{G}. Next, Molina obtained the description of the Chow ring of $Spin_8$ over complex numbers in \cite{MR}. On the other hand, Yagita computed in \cite{Y2} the whole motivic cohomology with $\Z/2$-coefficients for $Spin_7$ and $G_2$ and provided a bound for the Chow ring with $\Z/2$-coefficients of all $Spin_n$ over complex numbers in \cite{Y} by exploiting Quillen's computation of the singular cohomology of $BSpin_n$. In this paper we obtain a similar result by exploiting instead our computation of the motivic cohomology of the Nisnevich classifying space of $Spin_n$ which allows to relax the hypothesis on the base field and also suggests that understanding Nisnevich classifying spaces can possibly help in the study of the \'etale ones over more general fields.\\
	
	\textbf{Outline.} We now shortly summarise the content of each section of this text. In Sections $2$ and $3$ we give some notations that we follow throughout this paper and recall some preliminary results from \cite{Q} regarding the computation of the cohomology ring of $BSpin_n$ in topology. In Section $4$ we present some definitions and properties of the category of motives over a simplicial scheme which provide us with the main techniques essential to deal with fibrations of simplicial schemes with motivically Tate fibers. Section $5$ is devoted to Nisnevich classifying spaces, to show some of their features and, in particular, to recall subtle Stiefel-Whitney classes. In Section $6$ we construct a grid of long exact sequences involving the motivic cohomology of $BSpin_n$ and of $BSO_n$ which is our key tool, substituting the Serre spectral sequence, to get our main result. In Section $7$ we show some results about regular sequences in $H(BSO_n)$ obtained by acting with some Steenrod operations on the second subtle Stiefel-Whitney class, which allows us in Section $8$ to prove the main theorem, i.e. the computation of the motivic cohomology ring of $BSpin_n$. We see that, in general, this is not polynomial anymore in subtle Stiefel-Whitney classes, since many non trivial relations appear among them related to the action of the motivic Steenrod algebra on the second subtle Stiefel-Whitney class, and new subtle classes appear. Section $9$ is devoted to the computation of the motivic cohomology ring of $BG_2$. In Sections $10$ and $11$, using previous results, we find very simple relations among subtle classes in the motivic cohomology rings of \v{C}ech simplicial schemes associated to $Spin$-torsors and get some information about the Chern subring of the Chow ring with $\Z/2$-coefficients of the \'etale classifying space of $Spin_n$.\\
	
		\textbf{Acknowledgements.} I would like to thank my PhD supervisor Alexander Vishik for his support and many helpful advice that made this text possible. I also want to thank the referee for very useful remarks that helped to improve the exposition.
	
	\section{Notation}
	
	Let us start in this section by fixing some notations we will use throughout this paper.
	
	\begin{tabular}{c|c}
		$k$ & field of characteristic not $2$ containing $\sqrt{-1}$\\
		$Spc(k)$, $Spc_*(k)$ & category of motivic spaces over $k$, its pointed version\\
		${\mathcal H}_s(k)$, ${\mathcal H}_{s,*}(k)$ & simplicial homotopy category, its pointed version\\
		${\mathcal H}_{A^1}(k)$, ${\mathcal H}_{A^1,*}(k)$ & $A^1$-homotopy category of Morel-Voevodsky, its pointed version\\
		${\mathcal {DM}}^{-}_{eff}(k)$ & triangulated category of effective motives\\
		$T$ & unit Tate motive in ${\mathcal {DM}}^{-}_{eff}(k)$\\
		$H_{top}(-)$ & singular cohomology with $\Z/2$-coefficients\\
		$H(-)$ & motivic cohomology with $\Z/2$-coefficients\\
		$H$ & motivic cohomology with $\Z/2$-coefficients of $Spec(k)$\\
		$K^M(k)/2$ & Milnor K-theory of $k$ modulo $2$\\
		$w_i$ & $i$-th Stiefel-Whitney class in $H_{top}(BSO_n)$\\
		$u_i$ & $i$-th subtle Stiefel-Whitney class in $H(BSO_n)$\\
		$\rho_j$ & the element $Sq^{2^{j-1}}Sq^{2^{j-2}} \dots Sq^2Sq^1w_2$ in $H_{top}(BSO_n)$\\
		$\theta_j$ & the element $Sq^{2^{j-1}}Sq^{2^{j-2}} \dots Sq^2Sq^1u_2$ in $H(BSO_n)$\\
		$I^{top}_j$ & ideal in $H_{top}(BSO_n)$ generated by $\rho_0,\dots,\rho_{j-1}$\\
		$I_j$ & ideal in $H(BSO_n)$ generated by $\theta_0,\dots,\theta_{j-1}$\\
		$k(n)$ & $\max\{j:\rho_0,\dots,\rho_{j-1}$ is a regular sequence in $H_{top}(BSO_n)\}$ (see Theorem \ref{Q1})\\
		$h(n)$ & $\max\{j:\tau,\theta_0,\dots,\theta_{j-1}$ is a regular sequence in $H(BSO_n)\}$ (see Theorem \ref{seq})\\
		$v_{2^{k(n)}}$ & extra polynomial generator in $H(BSpin_n)$\\
	\end{tabular}\\
	
	It follows from results by Voevodsky (see \cite[Theorem 6.1, Corollary 6.9 and Corollary 7.5]{V3}) that $H \cong K^M(k)/2[\tau]$, where $\tau$ is the generator of $H^{0,1} \cong \Z/2$. At this point, recall from \cite[Lemma 11.1]{V2} and \cite[Theorem 1.1]{HKO} that the motivic Steenrod algebra is generated as a left $H$-module by the admissible monomials $Sq^{i_r} \dots Sq^{i_0}$ where $i_{j+1} \geq 2i_j \geq 0$. Each Steenrod square $Sq^i$ has bidegree $([i/2])[i]$, therefore $Sq^i(x)=0$ for $i>0$ and for any $x \in H^{n,n} \cong K_{n}^M(k)/2$, since $H$ is trivial above the diagonal. Moreover, since we are working over a field containing the square root of $-1$, we have that $Sq^1\tau=\rho=0$ where $\rho$ is the class of $-1$ in $K^M(k)/2$ and $Sq^i(\tau)=0$ for any $i \geq 2$ by \cite[Lemma 9.9]{V2}. It follows from this remark that, in our case, the only motivic cohomology operations that act non-trivially on $H$ are the multiplications by elements of $H$.
	
	\section{Preliminary results} 
	
	Our goal is to compute the motivic cohomology ring of the Nisnevich classifying space of $Spin_n$, the spin group of the standard split quadratic form $q_n$. In topology, the computation of the singular cohomology of $BSpin_n$ associated to the real euclidean quadratic form was achieved by Quillen in $\cite{Q}$.
	
	Before recalling his main results, let us define the elements $\rho_j$ in $H_{top}(BSO_n)\cong \Z/2[w_2,{\dots},w_n]$ inductively by the following formulas:
	$$\rho_0=w_2;$$
	$$\rho_{j+1}=Sq^{2^j}\rho_j.$$
	
	\begin{thm}\label{Q1}
		The sequence $\rho_0,\dots,\rho_{k(n)-1}$ is regular in $H_{top}(BSO_n)$, where $k(n)$ depends on $n$ as in the following table. 
		\begin{center}
			\begin{tabular}{|c|c|}
				\hline
				\textbf{n} &\textbf{k(n)}\\
				\hline
				8l+1 &4l\\
				\hline
				8l+2 &4l+1\\
				\hline
				8l+3 &4l+2\\
				\hline
				8l+4 &4l+2\\
				\hline
				8l+5 &4l+3\\ 
				\hline
				8l+6 &4l+3\\ 
				\hline
				8l+7 &4l+3\\ 
				\hline
				8l+8 &4l+3\\ 
				\hline
			\end{tabular}
		\end{center}
	\end{thm}
\begin{proof}
See \cite[Theorem 6.3]{Q}. 
\end{proof}
	
	Moreover, we recall that the values written in the previous table are related to the dimension of spin representations of $Spin_n$. More precisely, for any $n$ there is a spin representation $\Delta_n: Spin_n \rightarrow SO_{2^{k(n)}}$ that induces a map $B\Delta_n: BSpin_n \rightarrow BSO_{2^{k(n)}}$ on classifying spaces which, in turn, induces a homomorphism in cohomology $B\Delta_n^*: H_{top}(BSO_{2^{k(n)}})\rightarrow H_{top}(BSpin_n )$. We denote by $w_i(\Delta_n)$ the cohomology class $B\Delta_n^*(w_i)$ in $H_{top}(BSpin_n )$.
	
	\begin{thm}\label{Q2}
		 Let $I_{k(n)}^{top}$ be the ideal in $H_{top}(BSO_n)$ generated by the regular sequence from Theorem \ref{Q1}. Then, the canonical homomorphism 
		$$H_{top}(BSO_n)/I_{k(n)}^{top} \otimes \Z/2[w_{2^{k(n)}}(\Delta_n)] \rightarrow H_{top}(BSpin_n)$$ 
		is an isomorphism.
	\end{thm} 
\begin{proof}
	See \cite[Theorem 6.5]{Q}. 
\end{proof}
	
	\begin{rem}\label{rem1}
		\normalfont
	From Theorem \ref{Q1} and Theorem \ref{Q2} it follows that 
	$$k(n+1)=
	\begin{cases}
	k(n), & \rho_{k(n)} \in I_{k(n)}^{top} \\
	k(n)+1, & \rho_{k(n)} \notin I_{k(n)}^{top}
	\end{cases}
	$$
	where here by $I_{k(n)}^{top}$ we mean the ideal in $H_{top}(BSO_{n+1}) \cong \Z/2[w_2,\dots,w_{n+1}]$ generated by the elements $\rho_0,\dots,\rho_{k(n)-1}$.
	\end{rem}
	
	Furthermore, we notice that Theorem \ref{Q2} relies on the Serre spectral sequence for the fibration $B\Z/2 \rightarrow BSpin_n \rightarrow BSO_n$. In the motivic setting we do not have such a tool, so we use instead techniques developed by Smirnov and Vishik in $\cite{SV}$ which we recall in the following sections.
	
	\section{Motives over a simplicial base}
	
	The main purpose of this section is to recall some key definitions and results regarding the triangulated category of motives over a simplicial base, which is an essential tool for our computation. Before starting, we would like to mention that the contents of this section are essentially the same as Section 3 in \cite{T}. Here, there is only a further attention in the construction of all cofiber sequences at the level of motivic spaces first, which is needed for the compatibility with Steenrod operations. Moreover, there is the definition of Thom class and Corollary \ref{Thom3}, which were not present in \cite{T}.
	
	Let us fix a smooth simplicial scheme $Y_{\bullet}$ over $k$ and a commutative ring with identity $R$. Following \cite{V1}, we denote by $Sm/Y_{\bullet}$ the category in which objects are given by pairs $(U,j)$, with $j$ a non-negative integer and $U$ a smooth scheme over $Y_j$, and in which morphisms from $(U,j)$ to $(V,i)$ are given by pairs $(f,\theta)$, with $\theta:[i] \rightarrow [j]$ a simplicial map and $f:U \rightarrow V$ a morphism of schemes, such that the following diagram is commutative
	$$
	\xymatrix{
		U \ar@{->}[r]^{f} \ar@{->}[d] & V \ar@{->}[d]\\
		Y_j \ar@{->}[r]_{Y_{\theta}} & Y_i
	.}
	$$
	
	Moreover, as for spaces over the point, let us denote by $Spc(Y_{\bullet})=\Delta^{op}Shv_{Nis}(Sm/Y_{\bullet})$ the category of motivic spaces over $Y_{\bullet}$ and by $Spc_*(Y_{\bullet})$ its pointed counterpart, consisting of simplicial Nisnevich sheaves over $Sm/Y_{\bullet}$.
	
	For any morphism $f:U \rightarrow V$ in $Spc_*(Y_{\bullet})$ there is a cofiber sequence
	$$U \rightarrow V \rightarrow Cone(f) \rightarrow S^1 \wedge U$$
	where $Cone(f)$ is defined by the following push-out diagram in $Spc_*(Y_{\bullet})$
	$$
	\xymatrix{
		U \ar@{->}[r]^{f} \ar@{->}[d] & V \ar@{->}[d]\\
		U \wedge \Delta[1] \ar@{->}[r] & Cone(f).
	}
	$$
	
	In $\cite{V1}$ there is a construction of the category of motives over $Y_{\bullet}$ with $R$-coefficients. This category is denoted by ${\mathcal {DM}}_{eff}^-(Y_{\bullet},R)$. We notice that every cofiber sequence in $Spc_*(Y_{\bullet})$ induces a distinguished triangle in ${\mathcal {DM}}_{eff}^-(Y_{\bullet},R)$. Besides, attached to this category there is a sequence of restriction functors 
	$$r_i^*:{\mathcal {DM}}_{eff}^-(Y_{\bullet},R) \rightarrow {\mathcal {DM}}_{eff}^-(Y_i,R).$$
	
	The image of a motive $N \in {\mathcal {DM}}_{eff}^-(Y_{\bullet},R)$ under $r_i^*$ is simply denoted by $N_i$. Furthermore, we have the following adjunction for any morphism $p:Y_{\bullet} \rightarrow Y'_{\bullet}$ of smooth simplicial schemes
	\begin{align*}
		{\mathcal {DM}}_{eff}^-&(Y_{\bullet},R)\\
		Lp^* \uparrow & \downarrow Rp_*\\
		{\mathcal {DM}}_{eff}^-&(Y'_{\bullet},R).
	\end{align*}
	
	In the case that $p$ is smooth, together with the previous one, there is also the following adjunction
	\begin{align*}
		{\mathcal {DM}}_{eff}^-&(Y_{\bullet},R)\\
		Lp_{\#} \downarrow & \uparrow p^*\\
		{\mathcal {DM}}_{eff}^-&(Y'_{\bullet},R).
	\end{align*}

In particular, for any smooth simplicial scheme $Y_{\bullet}$ over $k$, we have a pair of adjoint functors
\begin{align*}
	{\mathcal DM}_{eff}^-&(Y_{\bullet},R)\\
	Lc_{\#} \downarrow & \uparrow c^*\\
	{\mathcal DM}_{eff}^-&(k,R)
\end{align*}
where $c:Y_{\bullet} \rightarrow Spec(k)$ is the projection to the base. Then, following \cite[Section 5]{V1}, one can define Tate objects $T(q)[p]$ in ${\mathcal DM}^-_{eff}(Y_{\bullet},R)$ as $c^*(T(q)[p])$. 
	
	At this point, we recall some facts about coherence taken from \cite{SV}. By a smooth coherent morphism we mean a smooth morphism $\pi:X_{\bullet} \rightarrow Y_{\bullet}$ such that there is a cartesian diagram
	$$
	\xymatrix{
		X_j \ar@{->}[r]^{\pi_j} \ar@{->}[d]_{X_{\theta}} & Y_j \ar@{->}[d]^{Y_{\theta}}\\
		X_i \ar@{->}[r]_{\pi_i} & Y_i
	}
	$$
	for any simplicial map $\theta:[i] \rightarrow [j]$. A motive $N$ in ${\mathcal {DM}}_{eff}^-(Y_{\bullet},R)$ is said to be coherent if all simplicial morphisms $\theta:[i] \rightarrow [j]$ induce structural isomorphisms $N_{\theta}:LY_{\theta}^*(N_i) \rightarrow N_j$. The full subcategory of ${\mathcal {DM}}_{eff}^-(Y_{\bullet},R)$ whose objects are coherent motives is denoted by ${\mathcal {DM}}_{coh}^-(Y_{\bullet},R)$. The fact that $LY_{\theta}^*$ is a triangulated functor implies that ${\mathcal {DM}}_{coh}^-(Y_{\bullet},R)$ is closed under taking cones and arbitrary direct sums. On the other hand, we have that $L\pi_{\#}$ maps coherent objects to coherent ones for any smooth coherent morphism $\pi$. Hence, $M(X_{\bullet} \xrightarrow{\pi} Y_{\bullet})$ is a coherent motive, where by $M(X_{\bullet} \xrightarrow{\pi} Y_{\bullet})$ we mean the image $L\pi_{\#}(T)$ of the unit Tate motive.
	
	In the following results, $CC(Y_{\bullet})$ indicates the simplicial set built up from a simplicial scheme $Y_{\bullet}$ by applying the functor $CC$ sending any connected scheme to the point and commuting with coproducts.
	
	The following proposition permits us to deal with fibrations of simplicial schemes with motivically Tate fibers.
	
	\begin{prop}\label{SV}
		Let $Y_{\bullet}$ be a simplicial scheme, $R$ be a commutative ring
		with identity, and suppose that the first singular cohomology group $H^1(CC(Y_{\bullet}),R^{\times})$ is trivial. Let $r,s$ be non-negative integers, and let $N \in {\mathcal {DM}}_{coh}^-(Y_{\bullet},R)$ be a motive such that $N_i \cong T(r)[s]$ in ${\mathcal {DM}}_{eff}^-(Y_i,R)$ for all i. Then, $N \cong T(r)[s]$ in ${\mathcal {DM}}_{eff}^-(Y_{\bullet},R)$.
	\end{prop}
\begin{proof}
See \cite[Proposition 3.1.5]{SV}. 
\end{proof}
	
	We point out that, for $R=\Z/2$, the cohomology group $H^1(CC(Y_{\bullet}),R^{\times})$ is always trivial.
	
	The next result is the core technique inspired by \cite{SV} that enables to generate long exact sequences in motivic cohomology, similar to Gysin sequences for sphere bundles in topology, for fibrations with motivically Tate fibers.
	
	\begin{prop}\label{Thom1}
		Let $\pi:X_{\bullet} \rightarrow Y_{\bullet}$ be a smooth coherent morphism of smooth simplicial schemes over $k$ and $A$ a smooth $k$-scheme such that:\\
		1) over the $0$ simplicial component $\pi$ is the projection $Y_0 \times A \rightarrow Y_0$;\\
		2) $H^1(CC(Y_{\bullet}),R^{\times}) \cong 0$;\\
		3) $M(A) \cong T \oplus T(r)[s-1] \in {\mathcal {DM}}_{eff}^-(k,R)$.
		
		Then, $M(Cone(\pi)) \cong T(r)[s] \in {\mathcal {DM}}_{eff}^-(Y_{\bullet},R)$ where $Cone(\pi)$ is the cone of $\pi$ in $Spc_*(Y_{\bullet})$. Moreover, we get a Thom isomorphism of $H(Y_{\bullet},R)$-modules 
		$$H^{*-s,*'-r}(Y_{\bullet},R) \rightarrow H^{*,*'}(Cone(\pi),R).$$
	\end{prop}
	\begin{proof}
		In $Spc_*(Y_{\bullet})$ we have a cofiber sequence 
		$$X_{\bullet} \xrightarrow{\pi} Y_{\bullet} \rightarrow Cone(\pi) \rightarrow S^1 \wedge X_{\bullet}$$ 
		which induces a distinguished triangle 
		$$M(X_{\bullet} \xrightarrow{\pi} Y_{\bullet}) \rightarrow T \rightarrow M(Cone(\pi)) \rightarrow M(X_{\bullet} \xrightarrow{\pi} Y_{\bullet})[1]$$ 
		in the motivic category ${\mathcal {DM}}_{eff}^-(Y_{\bullet},R)$. Since $\pi$ is smooth coherent and $\pi_0$ is the projection $Y_0 \times A \rightarrow Y_0$ by hypothesis, we have that it is the projection over any simplicial component, i.e. $\pi_i$ is the projection $Y_i \times A \cong X_i \rightarrow Y_i$ for all $i$. It immediately follows that $M(X_i \xrightarrow{\pi_i} Y_i) \cong T \oplus T(r)[s-1]$ in ${\mathcal {DM}}_{eff}^-(Y_i,R)$ since by hypothesis $M(A) \cong T \oplus T(r)[s-1]$ in ${\mathcal {DM}}_{eff}^-(k,R)$. Hence, the map $\pi_i$ induces the projection $T \oplus T(r)[s-1] \rightarrow T$ in ${\mathcal {DM}}_{eff}^-(Y_i,R)$ for any $i$, from which we get that $M(Cone(\pi))_i \cong T(r)[s]$ in ${\mathcal {DM}}_{eff}^-(Y_i,R)$. Moreover, we point out that $M(Cone(\pi))$ is a coherent motive, since both $M(X_{\bullet} \xrightarrow{\pi} Y_{\bullet})$ and $T$ are coherent objects and ${\mathcal {DM}}_{coh}^-(Y_{\bullet},R)$ is closed under taking cones. Since we are also assuming by hypothesis that $H^1(CC(Y_{\bullet}),R^{\times}) \cong 0$ we can apply Proposition \ref{SV} to $M(Cone(\pi))$. Therefore, we obtain that $M(Cone(\pi)) \cong T(r)[s]$ in ${\mathcal {DM}}_{eff}^-(Y_{\bullet},R)$, and the proof is complete. 
	\end{proof}
	
	The image of $1$ under the Thom isomorphism is called \textit{Thom class} and it is denoted by $\alpha$. 
	
	Later on, we will also need the following proposition about functoriality of the Thom isomorphism.
	
	\begin{prop}\label{Thom2}
		Let $\pi:X_{\bullet} \rightarrow Y_{\bullet}$ and $\pi':X'_{\bullet} \rightarrow Y'_{\bullet}$ be smooth coherent morphisms of smooth simplicial schemes over $k$ with $Y_0$ connected and $A$ a smooth $k$-scheme that satisfies all conditions from the previous proposition with respect to $\pi'$ and such that the following diagram is cartesian with all morphisms smooth
		$$
		\xymatrix{
			X_{\bullet} \ar@{->}[r]^{\pi} \ar@{->}[d]_{p_X} & Y_{\bullet} \ar@{->}[d]^{p_Y}\\
			X'_{\bullet} \ar@{->}[r]_{\pi'} & Y'_{\bullet}.
		}
		$$
		
		Then, the induced square of motives in the category ${\mathcal {DM}}_{eff}^-(Y'_{\bullet},R)$ extends uniquely to a morphism of triangles where $Lp_{Y\#}M(Cone(\pi)) \rightarrow M(Cone(\pi'))$ is given by $M(p_Y)(r)[s]$.
	\end{prop}
	\begin{proof}
		We start by noticing that in $Spc_*(Y'_{\bullet})$ we can complete our commutative diagram to a morphism of cofiber sequences
		$$
		\xymatrix{
			X_{\bullet} \ar@{->}[r]^{\pi} \ar@{->}[d]_{p_X} & Y_{\bullet} \ar@{->}[r] \ar@{->}[d]^{p_Y} & Cone(\pi) \ar@{->}[r] \ar@{->}[d]^{p} & S^1{\wedge}X_{\bullet} \ar@{->}[d]^{id{\wedge}p_X}\\
			X'_{\bullet} \ar@{->}[r]^{\pi} & Y'_{\bullet} \ar@{->}[r] & Cone(\pi') \ar@{->}[r] & S^1{\wedge}X'_{\bullet}
		}
		$$
		which induces a morphism of distinguished triangles in ${\mathcal {DM}}_{eff}^-(Y'_{\bullet},R)$
		$$
		\xymatrix{
			Lp_{Y\#}M(X_{\bullet} \xrightarrow{\pi} Y_{\bullet}) \ar@{->}[r] \ar@{->}[d]_{M(p_X)} & Lp_{Y\#}T \ar@{->}[r] \ar@{->}[d]^{M(p_Y)} & Lp_{Y\#}Cone(\pi) \cong Lp_{Y\#}T(r)[s] \ar@{->}[r] \ar@{->}[d]^{M(p)} & Lp_{Y\#}M(X_{\bullet} \xrightarrow{\pi} Y_{\bullet})[1] \ar@{->}[d]^{M(p_X)[1]}\\
			M(X'_{\bullet} \xrightarrow{\pi'} Y'_{\bullet}) \ar@{->}[r] & T \ar@{->}[r] & Cone(\pi') \cong T(r)[s] \ar@{->}[r] & M(X'_{\bullet} \xrightarrow{\pi'} Y'_{\bullet})[1]
		}
		$$
		where the isomorphisms in the third column follow by Proposition \ref{Thom1}. If we restrict our previous diagrams to the $0$ simplicial component we obtain in $Spc_*(Y'_0)$
		$$
		\xymatrix{
			Y_0 \times A \ar@{->}[r]^{\pi_0} \ar@{->}[d]_{p_{Y_0} \times id} & Y_0 \ar@{->}[r] \ar@{->}[d]^{p_{Y_0}} & Cone(\pi_0) \ar@{->}[r] \ar@{->}[d]^{p_0} & S^1 \wedge (Y_0 \times A) \ar@{->}[d]^{id \wedge (p_{Y_0} \times id)}\\
			Y'_0 \times A \ar@{->}[r]^{\pi_0} & Y'_0 \ar@{->}[r] & Cone(\pi'_0) \ar@{->}[r] & S^1 \wedge (Y'_0 \times A)
		}
		$$
		and in ${\mathcal {DM}}_{eff}^-(Y'_0,R)$
		$$
		\xymatrix{
			Lp_{Y_0\#}T \oplus Lp_{Y_0\#}T(r)[s-1] \ar@{->}[r] \ar@{->}[d]_{M(p_{Y_0}) \oplus M(p_{Y_0})(r)[s-1]} & Lp_{Y_0\#}T \ar@{->}[r] \ar@{->}[d]^{M(p_{Y_0})} & Lp_{Y_0\#}T(r)[s] \ar@{->}[r] \ar@{->}[d]^{M(p_0)} & Lp_{Y_0\#}T[1] \oplus Lp_{Y_0\#}T(r)[s] \ar@{->}[d]^{M(p_{Y_0})[1] \oplus M(p_{Y_0})(r)[s]}\\
			T \oplus T(r)[s-1] \ar@{->}[r] & T \ar@{->}[r] & T(r)[s] \ar@{->}[r] & T[1] \oplus T(r)[s]
		}
		$$
		from which we deduce that $M(p_0)$ must be $M(p_{Y_0})(r)[s]$. Notice that the morphisms $M(p)(-r)[-s]$ and $M(p_Y)$ are in 
		$$Hom_{{\mathcal {DM}}_{eff}^-(Y'_{\bullet},R)}(Lp_{Y\#}T,T) \cong Hom_{{\mathcal {DM}}_{eff}^-(Y_{\bullet},R)}(T,p_Y^*T) \cong Hom_{{\mathcal {DM}}_{eff}^-(Y_{\bullet},R)}(T,T) \cong R$$
		and, for the same reason, $M(p_0)(-r)[-s]=M(p_{Y_0})$ is in 
		$$Hom_{{\mathcal {DM}}_{eff}^-(Y'_0,R)}(Lp_{Y_0\#}T,T) \cong Hom_{{\mathcal {DM}}_{eff}^-(Y_0,R)}(T,p_{Y_0}^*T)
		\cong Hom_{{\mathcal {DM}}_{eff}^-(Y_0,R)}(T,T) \cong R.$$
		
		Since the homomorphism
		$$r_0^*:Hom_{{\mathcal {DM}}_{eff}^-(Y_{\bullet},R)}(T,T) \cong R \rightarrow Hom_{{\mathcal {DM}}_{eff}^-(Y_0,R)}(T,T) \cong R$$
		is the identity on $R$, we get that $M(p)=M(p_Y)(r)[s]$, as we aimed to show. 
	\end{proof}
	
	In particular, from the previous proposition it immediately follows the next corollary about functoriality of Thom classes.
	
	\begin{cor}\label{Thom3}
		Under the hypothesis of Proposition \ref{Thom2}, the homomorphism of $H(Y'_{\bullet},R)$-modules 
		$$p^*:H^{*,*'}(Cone(\pi'),R) \rightarrow H^{*,*'}(Cone(\pi),R)$$ 
		sends $\alpha'$ to $\alpha$, where $\alpha'$ and $\alpha$ are the respective Thom classes.
	\end{cor}
	
	\section{The Nisnevich classifying space}
	
	Throughout this paper, we are mainly interested in Nisnevich classifying spaces of linear algebraic groups over $Spec(k)$. In this section we recall some of their properties and relations with  \'etale classifying spaces. The contents of this section are similar to Section 4 in \cite{T}. The main difference resides on the fact that, in order to deal with the $Spin$-case, it is essential to weaken the hypothesis in Proposition 5.1 from ``is injective" (see \cite[Proposition 4.1]{T}) to ``has trivial kernel". Moreover, here we have added Corollary \ref{rat}, Proposition \ref{bso} and Corollary \ref{corWu}, which were not present in \cite{T}.
	
	Given a linear algebraic group $G$ over $k$, let us call by $EG$ the simplicial scheme defined on simplicial components by $(EG)_n=G^{n+1}$ with partial projections and partial diagonals as face and degeneracy maps respectively. The operation in $G$ induces a natural action on $EG$. Then, the Nisnevich classifying space $BG$ is obtained by taking the quotient respect to this action, i.e. $BG=EG/G$. Moreover, from the morphism of sites $\pi:(Sm/k)_{\acute{e}t} \rightarrow (Sm/k)_{Nis}$ we obtain the following adjunction 
	\begin{align*}
		{\mathcal H}_s((S&m/k)_{\acute{e}t})\\
		\pi^* \uparrow & \downarrow R\pi_*\\
		{\mathcal H}_s((S&m/k)_{Nis})
	\end{align*}
	where $\pi_*$ is the restriction to Nisnevich topology and $\pi^*$ is \'etale sheafification. Then, a definition of the \'etale classifying space of $G$ is provided by $B_{\acute{e}t}G=R\pi_*\pi^*BG$ . Although this definition presents \'etale classifying spaces as objects of ${\mathcal H}_s((Sm/k)_{Nis}$, there exists a geometric construction for their $A^1$-homotopy type (see $\cite{MV}$) obtained from a faithful representation $\rho:G \hookrightarrow GL(V)$ by taking the quotient respect to the diagonal action of $G$ on an open subscheme of an infinite-dimensional affine space $\oplus_{i=1}^{\infty}V$ where $G$ acts freely. 
	
	Now, let $H$ be an algebraic subgroup of $G$. Then, we can define two simplicial objects related to $BH$, namely a bisimplicial scheme $\widetilde{B}H=(EH \times EG)/H$ and a simplicial scheme $\widehat{B}H=EG/H$. We highlight that the obvious morphism of simplicial schemes $\pi:\widehat{B}H \rightarrow BG$ is trivial over each simplicial component with $G/H$-fibers. At this point, let us call by $\phi:\widetilde{B}H \rightarrow BH$ and $\psi:\widetilde{B}H \rightarrow \widehat{B}H$ the two natural projections. Notice that $\phi$ is always trivial over each simplicial component with contractible fiber $EG$, therefore an isomorphism in ${\mathcal H}_s(k)$. The behaviour of $\psi$ is somewhat different. Indeed, we need to impose a precise condition in order to make it an isomorphism. 
	
	\begin{prop}\label{BG1}
		If the map $Hom_{{\mathcal H}_s(k)}(Spec(R),B_{\acute{e}t}H) \rightarrow Hom_{{\mathcal H}_s(k)}(Spec(R),B_{\acute{e}t}G)$ has trivial kernel for any Henselian local ring $R$ over $k$, then $\psi$ is an isomorphism in ${\mathcal H}_s(k)$. In particular, $BH \cong \widehat{B}H$ in ${\mathcal H}_s(k)$.
	\end{prop}
	\begin{proof} 
		We start by noticing that the restriction of $\psi$ over any simplicial component is given by the morphism $(EH \times G^{n+1})/H \rightarrow G^{n+1}/H$. The simplicial scheme $(EH \times G^{n+1})/H$ is nothing but the \v{C}ech simplicial scheme $\check{C}({G^{n+1} \rightarrow G^{n+1}/H})$ associated to the $H$-torsor $G^{n+1} \rightarrow G^{n+1}/H$ which becomes split once extended to $G$. In order to check that
		$$\check{C}({G^{n+1} \rightarrow G^{n+1}/H}) \rightarrow G^{n+1}/H$$ 
		is a simplicial weak equivalence it is enough, by \cite[Lemma 1.11]{MV}, to evaluate on henselian local rings. Therefore, we need to look at the morphism of simplicial sets
		$$\check{C}({G^{n+1}(R) \rightarrow (G^{n+1}/H)(R)}) \rightarrow (G^{n+1}/H)(R)$$
		for any henselian local ring $R$ over $k$. Now, the fiber of $G^{n+1} \rightarrow G^{n+1}/H$ over any point $Spec(R)$ of $G^{n+1}/H$ is given by a $H$-torsor $P \rightarrow Spec(R)$ whose extension to $G$ is split, so split itself by hypothesis. In other words, this fiber is nothing but the split $H$-torsor $H \times Spec(R) \rightarrow Spec(R)$. In this way we have found a splitting of $G^{n+1}(R) \rightarrow (G^{n+1}/H)(R)$ which proves that $\check{C}({G^{n+1}(R) \rightarrow (G^{n+1}/H)(R)}) \rightarrow (G^{n+1}/H)(R)$ is a weak equivalence of simplicial sets, for any henselian local ring $R$. This implies that $\psi$ is an isomorphism in ${\mathcal H}_s(k)$. 
	\end{proof}
	
	In practice, in the case we are interested in, it is enough to check the hypothesis of the previous proposition only for field extensions of $k$. The reason resides on the fact that rationally trivial quadratic forms are Zariski-locally trivial (see \cite[Theorem 5.1]{OP}). Indeed, we have the following corollary to the previous proposition.
	
	\begin{cor}\label{rat}
		Let $H$ and $G$ be such that all rationally trivial $H$-torsors and $G$-torsors are Zariski-locally trivial. If the map $Hom_{{\mathcal H}_s(k)}(Spec(K),B_{\acute{e}t}H) \rightarrow Hom_{{\mathcal H}_s(k)}(Spec(K),B_{\acute{e}t}G)$ has trivial kernel for any field extension $K$ of $k$, then $\psi$ is an isomorphism in ${\mathcal H}_s(k)$. In particular, $BH \cong \widehat{B}H$ in ${\mathcal H}_s(k)$.
	\end{cor}
	\begin{proof}
		Let $R$ be any Henselian local ring over $k$ and $K$ its field of fractions. Then, we have the following commutative diagram
		$$
		\xymatrix{
			Hom_{{\mathcal H}_s(k)}(Spec(R),B_{\acute{e}t}H) \ar@{->}[r] \ar@{->}[d] & Hom_{{\mathcal H}_s(k)}(Spec(R),B_{\acute{e}t}G) \ar@{->}[d]\\
			Hom_{{\mathcal H}_s(k)}(Spec(K),B_{\acute{e}t}H) \ar@{->}[r] & Hom_{{\mathcal H}_s(k)}(Spec(K),B_{\acute{e}t}G).
		}
		$$
		
		Saying that all rationally trivial $H$-torsors and $G$-torsors are Zariski-locally trivial implies that the two vertical maps in the previous diagram have trivial kernels. Moreover, by hypothesis, we have that the bottom horizontal map has trivial kernel too. Therefore, the top horizontal map has trivial kernel and the statement follows by Proposition \ref{BG1}. 
	\end{proof}
	
	The natural embedding of algebraic groups $H \hookrightarrow G$ induces two morphisms $j:BH \rightarrow \widehat{B}H$ and $g:BH \rightarrow BG$. The following result tells us that, under the hypothesis of the previous proposition, $j$ identifies $BH$ and $\widehat{B}H$ in ${\mathcal H}_s(k)$.
	
	\begin{prop}\label{BG2}
		Under the hypothesis of Proposition \ref{BG1}, $j$ is an isomorphism in ${\mathcal H}_s(k)$.
	\end{prop}
	\begin{proof}
		We already know that in this case the morphisms of bisimplicial schemes $\phi$ and $\psi$ become weak equivalences once restricted to simplicial components. It follows that the morphisms they induce on the respective diagonal simplicial objects, namely $\phi:\Delta(\widetilde{B}H) \rightarrow BH$ and $\psi:\Delta(\widetilde{B}H) \rightarrow \widehat{B}H$, are weak equivalences. So, in order to get the result, it is enough to provide a simplicial homotopy $F_i^{(n)}:(H^{n+1} \times G^{n+1})/H \rightarrow G^{n+2}/H$ between $j \phi$ and $\psi$. One is given by 
		$$F_i^{(n)}(h_0,{\dots},h_n,g_0,{\dots},\gamma_n)=(h_0,{\dots},h_i,g_i,{\dots},\gamma_n)$$ 
		for any $n$ and any $0 \leq i \leq n$.  
	\end{proof}
	
	\begin{rem}\label{hom}
		\normalfont
	Note that, since $g=\pi j$, the homomorphism $j^*:H(\widehat{B}H) \rightarrow H(BH)$ is an isomorphism of $H(BG)$-modules.
	\end{rem}
	
	The reason why we would like to work with $\widehat{B}H \rightarrow BG$ instead of $BH \rightarrow BG$ is that the first is a coherent morphism which is trivial over the $0$ simplicial component with fiber $G/H$. So, provided that the reduced motive of $G/H$ is Tate, we could apply to it Proposition \ref{Thom1}. In a nutshell, this is how one can reconstruct the cohomology of the Nisnevich classifying space of an algebraic group inductively by considering some filtration of it.
	
	We now move our attention to some particular examples which are of main interest for the purposes of this paper. First, we recall that $O_n$-torsors are in one-to-one correspondence with quadratic forms, $SO_n$-torsors are in one-to-one correspondence with quadratic forms with trivial discriminant and $Spin_n$-torsors yield quadratic forms with trivial discriminant and Clifford invariant via a surjective map with trivial kernel for $n \geq 2$. Hence, we can apply Propositions \ref{BG1} and \ref{BG2} to the case that $G$ and $H$ are respectively $O_{n+1}$ and $O_n$, or $SO_{n+1}$ and $SO_n$, or $Spin_{n+1}$ and $Spin_n$ for $n \geq 2$.  Moreover, we have the following short exact sequences of algebraic groups
	$$1 \rightarrow SO_n \rightarrow O_n \rightarrow \mu_2 \rightarrow 1,$$
	$$1 \rightarrow \mu_2 \rightarrow Spin_n \rightarrow SO_n \rightarrow 1$$
	from which we get that 
	$$A_{q_{n+1}} \cong O_{n+1}/O_n \cong SO_{n+1}/SO_n \cong Spin_{n+1}/Spin_n$$
	where $A_{q_{n+1}}$ is the affine quadric defined by the equation $q_{n+1}=1$. Now we recall that 
	$$M(A_{q_{n+1}}) \cong \Z \oplus \Z([(n+1)/2])[n] \in DM^-_{eff}(k)$$ 
	by \cite[Proposition 3.1.3]{SV}. Hence, we can apply Proposition \ref{Thom1} to the fibrations we are mostly interested in, namely $\widehat{B}O_n \rightarrow BO_{n+1}$, $\widehat{B}SO_n \rightarrow BSO_{n+1}$ and $\widehat{B}Spin_n \rightarrow BSpin_{n+1}$.
	
	Indeed, by exploiting the arguments above mentioned the following theorem is obtained in \cite{SV}.
	
	\begin{thm}\label{SubtleSW}
		 There is a unique set $u_1,{\dots},u_n$ of classes in the motivic $\Z/2$-cohomology of $BO_n$ such that $deg(u_i)=([i/2])[i]$, $u_i$ vanishes when restricted to $H(BO_{i-1})$ for any $2 \leq i \leq n$ and
		$$H(BO_n) \cong H[u_1,{\dots},u_n].$$
	\end{thm}
\begin{proof}
See \cite[Theorem 3.1.1]{SV}. 
\end{proof}
	
	The generators $u_i$ are called \textit{subtle Stiefel-Whitney classes}. It is possible to get the same description for $H(BSO_n)$ with the only difference given by the fact that $u_1=0$. Indeed, one has the following result.
	
	\begin{prop}\label{bso}
		The motivic cohomology ring of $BSO_n$ is completely described by
		$$H(BSO_n) \cong H[u_2,\dots,u_n].$$
	\end{prop}
	\begin{proof}
		It is enough to apply Proposition \ref{Thom1} to the coherent morphism $\widehat{B}SO_n \rightarrow BO_n$ whose fiber is isomorphic to $\mu_2$. This way one gets a Gysin long exact sequence of $H(BO_n)$-modules in motivic cohomology
		$$\dots \rightarrow H^{p-1,q}(BO_n) \xrightarrow{j^*} H^{p,q}(BO_n) \xrightarrow{k^*} H^{p,q}(BSO_n)\xrightarrow{l^*} H^{p,q}(BO_n) \rightarrow \dots .$$
		
		Now, note that $k^*$ is a ring homomorphism, hence it sends $1$ to $1$. Since $H^{0,0}(BSO_n) \cong \Z/2$, it follows that $l^*$ is the $0$ homomorphism in bidegree $(0)[0]$. This implies that $j^*$ sends $1$ to $u_1$. From the fact that it is a homomorphism of $H(BO_n)$-modules we deduce that $j^*$ is the multiplication by $u_1$. Hence, it is a monomorphism in all bidegrees, from which it follows that $l^*$ is the $0$ homomorphism in all bidegrees. Therefore, $k^*$ is an epimorphism and it kills all monomials divisible by $u_1$, from which we deduce that $H(BSO_n) \cong H[u_2,\dots,u_n]$. 
	\end{proof}
	
	Unfortunately, as we will see, while for orthogonal and special orthogonal groups Gysin sequences are enough to get the description of the motivic cohomology of their classifying spaces, for spin groups this is not true anymore. Indeed, we need to use also the fibrations $BSpin_n \rightarrow BSO_n$ and study their induced homomorphisms in cohomology. We achieve this in the following sections.
	
	We will also use the action of the motivic Steenrod algebra on subtle classes which is given by the following Wu formula as in the classical case.
	
	\begin{prop}\label{Wu}
		$$Sq^ku_m=
		\begin{cases}
		\sum_{j=0}^k{m+j-k-1 \choose j}u_{k-j}u_{m+j}, & 0 \leq k<m \\
		u_m^2, & k=m \\
		0, & k>m
		\end{cases}
		$$
	\end{prop}
\begin{proof}
See \cite[Proposition 3.1.12]{SV}. 
\end{proof}
	
	From the previous result we immediately deduce the following corollary which will be useful in the next sections.
	
	\begin{cor}\label{corWu}
		Let $w$ be a monomial of bidegree $([{\frac m 2}])[m]$ in $\Z/2[\tau,u_2,{\dots},u_{n+1}]$. Then, $Sq^mw=w^2$ and $Sq^jw=0$ for any $j>m$. 
	\end{cor}
	\begin{proof}  
		If $m$ is even, by \cite[Lemmas 9.8 and 9.9]{V2}, there is nothing to prove since $w$ is on the slope $2$ diagonal. Consider $m$ odd, then $w=\tau^{{\frac {r-1} 2}}xu_{h_1}{\cdots}u_{h_r}$ where $x$ is a monomial in even subtle classes and $u_{h_i}$ are odd subtle classes (notice that $r$ must be odd by degree reason). Therefore, by Cartan formula, we have that $Sq^mw=Sq^m(\tau^{{\frac {r-1} 2}}xu_{h_1}{\cdots}u_{h_r})=Sq^{m-h_r}(\tau^{{\frac {r-1} 2}}xu_{h_1}{\cdots}u_{h_{r-1}})Sq^{h_r}u_{h_r}=(\tau^{{\frac {r-1} 2}}xu_{h_1}{\cdots}u_{h_{r-1}})^2u_{h_r}^2=w^2$ since the monomial $\tau^{{\frac {r-1} 2}}xu_{h_1}{\cdots}u_{h_{r-1}}$ is on the slope $2$ diagonal. Moreover, $Sq^jw=0$ for $j>m$ for the same reason. 
	\end{proof}
	
	\section{The fibration $BSpin_n \rightarrow BSO_n$}\label{D}
	
	We have already noticed that the special orthogonal case does not differ much from the orthogonal one, at least from the cohomological perspective, in the sense that their motivic cohomology rings are both polynomial over the cohomology of the point in subtle Stiefel-Whitney classes. This is not true anymore for spin groups. The main reason is that in this case there are much more complicated relations among subtle classes given by the action of the motivic Steenrod algebra on $u_2$ which make the cohomology rings not polynomial in subtle Stiefel-Whitney classes anymore (precisely for $n>9$) and, moreover, new classes appear. For this reason, in order to get our main result, together with an inductive argument we need to consider the fibration $BSpin_n \rightarrow BSO_n$. More precisely, in order to investigate the motivic cohomology of $BSpin_n$, we need to consider for any $n \geq 2$ the commutative square
	$$
	\xymatrix{
		\widehat{B}Spin_n \ar@{->}[r]^{\widehat{a}_n} \ar@{->}[d]_{\widetilde{\pi}} & \widehat{B}SO_n \ar@{->}[d]^{\pi}\\
		BSpin_{n+1} \ar@{->}[r]^{a_{n+1}} & BSO_{n+1}
	}
	$$
	where $\pi$ and $\widetilde{\pi}$ are smooth coherent morphisms, trivial over simplicial components, with fiber isomorphic to the affine quadric $A_{q_{n+1}}$ defined by the equation $q_{n+1}=1$.
	
	In $Spc_*(BSO_{n+1})$ we can complete the previous diagram to the following one (commutative up to a sign in the right bottom square) where each row and each column is a cofiber sequence
	$$
	\xymatrix{
		\widehat{B}Spin_n \ar@{->}[r]^{\widehat{a}_n} \ar@{->}[d]_{\widetilde{\pi}} & \widehat{B}SO_n \ar@{->}[r]^{\widehat{b}_n} \ar@{->}[d]^{\pi} & Cone(\widehat{a}_n) \ar@{->}[r]^{\widehat{c}_n} \ar@{->}[d]^{\overline{\pi}} & S^1 \wedge \widehat{B}Spin_n \ar@{->}[d]\\
		BSpin_{n+1} \ar@{->}[r]^{a_{n+1}} \ar@{->}[d]_{\widetilde{f}} & BSO_{n+1} \ar@{->}[r]^{b_{n+1}} \ar@{->}[d]^{f} & Cone(a_{n+1}) \ar@{->}[r]^{c_{n+1}} \ar@{->}[d]^{\overline{f}} & S^1 \wedge BSpin_{n+1} \ar@{->}[d]\\
		Cone(\widetilde{\pi}) \ar@{->}[r] \ar@{->}[d] & Cone(\pi) \ar@{->}[r] \ar@{->}[d] & Cone(\overline{\pi}) \ar@{->}[r] \ar@{->}[d] & S^1 \wedge Cone(\widetilde{\pi}) \ar@{->}[d]\\
		S^1 \wedge \widehat{B}Spin_n \ar@{->}[r] & S^1 \wedge \widehat{B}SO_n \ar@{->}[r] & S^1 \wedge Cone(\widehat{a}_n) \ar@{->}[r] & S^2 \wedge \widehat{B}Spin_n.\\
	}
	$$
	
	The previous diagram induces, in turn, a commutative diagram of long exact sequences in motivic cohomology with $\Z/2$-coefficients, where all the homomorphisms are compatible with Steenrod operations and respect the $H(BSO_{n+1})$-module structure. This remark comes from the fact that the following diagram of categories
	$$
	\xymatrix{
		Spc_*(BSO_{n+1}) \ar@{->}[r] \ar@{->}[d] & {\mathcal H}_{A^1,*}(k) \ar@{->}[d]\\
		{\mathcal {DM}}_{eff}^-(BSO_{n+1},\Z/2) \ar@{->}[r] & {\mathcal {DM}}_{eff}^-(k,\Z/2)
	}
	$$
	is commutative up to a natural equivalence and both functors in the right bottom corner have adjoints from the right, so we have the action of Steenrod operations on the motivic cohomology of objects belonging to the image of $Spc_*(BSO_{n+1})$ in ${\mathcal {DM}}_{eff}^-(BSO_{n+1},\Z/2)$ pulled from ${\mathcal H}_{A^1,*}(k)$.
	
	Since $Spin$-torsors yield quadratic forms from $I^3$ via a map with trivial kernel and for quadratic forms Witt cancellation holds, we are allowed to use Propositions \ref{BG1} and \ref{BG2} and Remark \ref{hom}. As a result, we get the following infinite grid of long exact sequences
	\begin{align}\label{d1}
	\xymatrix{
		& \vdots \ar@{->}[d] & \vdots \ar@{->}[d] & \vdots \ar@{->}[d] & \vdots \ar@{->}[d]\\
		\dots \ar@{->}[r] & H^{p-2,q}(BSpin_n) \ar@{->}[r]^{c_n^*} \ar@{->}[d]_{\widetilde{h}^*} & H^{p-1,q}(Cone(a_n)) \ar@{->}[r]^{b_n^*} \ar@{->}[d]^{\overline{h}^*} & H^{p-1,q}(BSO_n) \ar@{->}[r]^{a_n^*} \ar@{->}[d]^{h^*} & H^{p-1,q}(BSpin_n) \ar@{->}[d] \ar@{->}[r] & \dots\\
		\dots \ar@{->}[r] & H^{p-1,q}(Cone(\widetilde{\pi})) \ar@{->}[r] \ar@{->}[d]_{\widetilde{f}^*} & H^{p,q}(Cone(\overline{\pi})) \ar@{->}[r] \ar@{->}[d]^{\overline{f}^*} & H^{p,q}(Cone(\pi)) \ar@{->}[r] \ar@{->}[d]^{f^*} & H^{p,q}(Cone(\widetilde{\pi})) \ar@{->}[d] \ar@{->}[r] & \dots\\
		\dots \ar@{->}[r] & H^{p-1,q}(BSpin_{n+1}) \ar@{->}[r]^{c_{n+1}^*} \ar@{->}[d]_{\widetilde{g}^*} & H^{p,q}(Cone(a_{n+1})) \ar@{->}[r]^{b_{n+1}^*} \ar@{->}[d]^{\overline{g}^*} & H^{p,q}(BSO_{n+1}) \ar@{->}[r]^{a_{n+1}^*} \ar@{->}[d]^{g^*} & H^{p,q}(BSpin_{n+1}) \ar@{->}[d] \ar@{->}[r] & \dots\\
		\dots \ar@{->}[r] & H^{p-1,q}(BSpin_n) \ar@{->}[r] \ar@{->}[d] & H^{p,q}(Cone(a_n)) \ar@{->}[r] \ar@{->}[d] & H^{p,q}(BSO_n) \ar@{->}[r] \ar@{->}[d] & H^{p,q}(BSpin_n) \ar@{->}[r] \ar@{->}[d] & \dots\\
		& \vdots  & \vdots & \vdots & \vdots}
	\end{align}
	where all the homomorphisms are compatible with Steenrod operations and respect the $H(BSO_{n+1})$-module structure.
	
	We recall that, by applying Proposition \ref{Thom1} to the smooth coherent morphism $\pi: \widehat{B}SO_n \rightarrow BSO_{n+1}$, which has fiber isomorphic to $A_{q_{n+1}}$ whose reduced motive is Tate, there is a Thom isomorphism 
	$$H^{p-n-1,q-[{\frac{n+1}{2}}]}(BSO_{n+1}) \rightarrow H^{p,q}(Cone(\pi))$$
	which sends $1$ to the Thom class $\alpha$. By Theorem \ref{SubtleSW}, modulo this isomorphism $f^*$ is just the multiplication by the subtle Stiefel-Whitney class $u_{n+1}$, since it is the only class of its bidegree vanishing in $H(BO_n)$. Since $Spin_{n+1}/Spin_n \cong A_{q_{n+1}}$, Proposition \ref{Thom1} applies also to the smooth coherent morphism $\widetilde{\pi}: \widehat{B}Spin_n \rightarrow BSpin_{n+1}$. Therefore, we have a Thom isomorphism 
	$$H^{p-n-1,q-[{\frac{n+1}{2}}]}(BSpin_{n+1}) \rightarrow H^{p,q}(Cone(\widetilde{\pi}))$$
	and a Thom class $\widetilde{\alpha} \in H^{n+1,[\frac{n+1}{2}]}(Cone(\widetilde{\pi}))$. We notice that, by Corollary \ref{Thom3}, $\widetilde{\alpha}$ is nothing but the restriction of $\alpha$ from $H^{n+1,[\frac{n+1}{2}]}(Cone(\pi))$ to $H^{n+1,[\frac{n+1}{2}]}(Cone(\widetilde{\pi}))$. Hence, modulo the Thom isomorphism, $\widetilde{f}^*$ is multiplication by $u_{n+1}$. Moreover, from Proposition \ref{Thom2} we have that the map $Cone(\widetilde{\pi}) \rightarrow Cone(\pi)$ induces in ${\mathcal {DM}}_{eff}^-(BSO_{n+1},\Z/2)$ the morphism 
	$$M(BSpin_{n+1} \xrightarrow{a_{n+1}} BSO_{n+1})([(n+1)/2])[n+1] \xrightarrow{M(a_{n+1})([(n+1)/2])[n+1]} T([(n+1)/2])[n+1]$$
	from which it follows that
	$$M(Cone(\overline{\pi})) \cong M(Cone(a_{n+1}))([(n+1)/2])[n+1]$$ 
	which induces an isomorphism 
	$$H^{p-n-1,q-[{\frac{n+1}{2}}]}(Cone(a_{n+1})) \rightarrow H^{p,q}(Cone(\overline{\pi})).$$ 
	
	Note that, from Theorem \ref{SubtleSW}, the morphism $h^*$ is always the $0$ homomorphism, which means at the same time that $g^*$ is surjective and $f^*$ is injective. From these remarks we obtain the next proposition.
	
	\begin{prop}\label{SteenrodThom}
		$Sq^m\alpha=u_m\alpha$ for any $m \leq n+1$ and $0$ otherwise. The same holds for $\widetilde{\alpha}$.
	\end{prop}
	\begin{proof}
		We just notice that $f^*(Sq^m\alpha)=Sq^mf^*(\alpha)=Sq^mu_{n+1}=u_mu_{n+1}=u_mf^*(\alpha)=f^*(u_m\alpha)$. The result follows by injectivity of $f^*$. 
	\end{proof}
	
	\section{Some regular sequences in $H(BSO_n)$}
	
	The main aim of this section is to prove a result in the motivic setting similar to Theorem \ref{Q1}. We construct a sequence $\theta_0,\theta_1,\dots,\theta_{{k(n)}-1}$ in $H(BSO_n)$ by applying some Steenrod operations to $u_2$ just as in the topological case. Then, we focus on the two sequences obtained from the previous one by imposing on the one hand $\tau=1$ and on the other $\tau=0$. While the regularity of the first sequence was completely established by Quillen, nothing was known about the regularity of the second. We follow Quillen's method which allows to obtain the regularity of the sequence in topology by studying it in the cohomology of a certain power of $BO_1$ where it has an easier shape, related to some quadratic form over $\Z/2$. The lenght $k(n)$ of the regular sequence essentially depends on the characteristics of this quadratic form. For $\tau=0$, this approach does not work completely, so we study instead our sequence in the cohomology of a certain power of $BO_2$. In this ring our sequence has a simple form, related now to a certain bilinear form over $\Z/2$. As for the topological case, by studying these bilinear forms, we are able to get the regularity of some sequences of lenght $h(n)$ (related to our initial motivic sequences) with $\tau=0$. Surprisingly, these sequences are either long as Quillen's sequences or have one less element. Then, combining Quillen's result ($\tau=1$) with ours ($\tau=0$), we get the regularity of $\theta_0,\theta_1,\dots,\theta_{{k(n)}-1}$ in the motivic cohomology of $BSO_n$ for the same values that appear in topology.
	
	Let $V$ be an $n$-dimensional $\Z/2$-vector space and $\Omega$ an algebraically closed field extension of $\Z/2$. We denote by $V_{\Omega}$ the $\Omega$-vector space $\Omega \otimes_{\Z/2} V$. Note that the Frobenius automorphism acts on $V_{\Omega}$ via the first tensor factor. Following \cite{Q}, we also denote by $x \mapsto x^2$ the Frobenius transformation on $V_{\Omega}$. First, we recall the following result from \cite{Q}.
	
	\begin{prop}\label{stable}
		An $\Omega$-subspace $M$ of $V_{\Omega}$ is of the form $W_{\Omega}$ for some subspace $W$ of $V$ if and only if $M$ is stable under the Frobenius transformation.
	\end{prop}
\begin{proof}
See \cite[Proposition 2.1]{Q}. 
\end{proof}
	
	Let $B$ be a bilinear form over $V$ and denote by ${^{\perp}V}$ its right radical, i.e.
	$${^{\perp}V}=\{y \in V: B(x,y)=0 \: for \: any \: x \in V\}.$$
	
	Note that $B(x,y)$ can be seen as a homogeneous element of degree $2$ in $\Omega[x_1,\dots,x_n,y_1,\dots,y_n]$. In fact, let $\{e_1,\dots,e_n\}$ be a basis for $V$, then $B(x,y)=\sum_{i,j=1}^n B(e_i,e_j)x_i y_j$ where the $x_i$ and the $y_j$ are the coordinates of $x$ and $y$ respectively in the chosen basis. Let $h=n-dim({^{\perp}V})$ and consider the ideal $J$ in $\Omega[x_1,\dots,x_n,y_1,\dots,y_n]$ generated by the homogeneous polynomials $B(x,y), B(x,y^2), \dots, B(x,y^{2^{h-1}})$.
	
	\begin{prop}
		The algebraic subset in $V_{\Omega} \times V_{\Omega}$ defined by the ideal $J$ is given by 
		$$Var(J)=\bigcup_{W \subseteq V} W_{\Omega}^{\perp} \times W_{\Omega}$$
		where $W_{\Omega}^{\perp}=\{x \in V_{\Omega}: B(x,y)=0 \: for \: any \: y \in W_{\Omega}\}$.
	\end{prop}
	\begin{proof}
		From Proposition \ref{stable} we know that $W_{\Omega}$ is stable under the Frobenius transformation for any subspace $W$ of $V$. Hence, if $(x,y)$ belongs to $W_{\Omega}^{\perp} \times W_{\Omega}$, then $y, y^2, \dots, y^{2^{h-1}}$ are in $W_{\Omega}$ and $x \in W_{\Omega}^{\perp}$. It follows that $B(x,y), B(x,y^2), \dots, B(x,y^{2^{h-1}})$ are all zero, so $(x,y) \in Var(J)$. Therefore, 
		$$\bigcup_{W \subseteq V} W_{\Omega}^{\perp} \times W_{\Omega} \subseteq Var(J).$$
		
		On the other hand, let $(x,y)$ be a point of $Var(J)$ and consider the subspace $M_y$ of $V_{\Omega}$ defined by 
		$$M_y=\langle y,y^2,\dots,y^{2^{h-1}}\rangle   + {^{\perp}V}_{\Omega}.$$
		
		Obviously, $(x,y)$ belongs to $M_y^{\perp} \times M_y$. In order to prove that $M_y$ is of the form $W_{\Omega}$ for some $W \subseteq V$ it is enough to show that $M_y$ is stable under the Frobenius transformation. Note that ${^{\perp}V}_{\Omega}$ is stable under the Frobenius transformation, and so, if $y^{2^i} \in {^{\perp}V}_{\Omega}$ for some $i$, then $y^{2^j} \in {^{\perp}V}_{\Omega}$ for all $j \geq i$. Hence, $M_y/ {^{\perp}V}_{\Omega}=\langle y,y^2,\dots,y^{2^{i-1}}\rangle  $ and $\langle y^{2^i},\dots,y^{2^{h-1}}\rangle   \subseteq {^{\perp}V}_{\Omega}$ for some $i \leq h$. If $i<h$, then $M_y$ is stable under the Frobenius transformation, since ${^{\perp}V}_{\Omega}$ is so. If $i=h$, then $M_y=\langle y,y^2,\dots,y^{2^{h-1}}\rangle   \oplus {^{\perp}V}_{\Omega}$. Therefore, if $y,y^2,\dots ,y^{2^{h-1}}$ are linearly independent then $M_y=V_{\Omega}$ since $dim({^{\perp}V})=n-h$, so $y^{2^h}$ clearly belongs to $M_y$. Otherwise, $y^{2^i} \in \langle y,\dots,y^{2^{i-1}}\rangle  $ for some $0 \leq i \leq h-1$, from which it follows that $y^{2^h} \in \langle y,\dots,y^{2^{h-1}}\rangle  $. Hence, $M_y$ is stable under the Frobenius transformation from which we deduce that
		$$Var(J) \subseteq \bigcup_{W \subseteq V} W_{\Omega}^{\perp} \times W_{\Omega}$$
		which completes the proof. 
	\end{proof}
	
	From the previous proposition we immediately obtain the following result.
	
	\begin{cor}\label{reg}
		The sequence $B(x,y), B(x,y^2), \dots, B(x,y^{2^{h-1}})$ is a regular sequence in the polynomial ring $\Z/2[x_1,\dots,x_n,y_1,\dots,y_n]$.
	\end{cor}
	\begin{proof}
		Recall that $r_1,\dots,r_h$ is a regular sequence in the polynomial ring  $\Omega[x_1,\dots,x_n,y_1,\dots,y_n]$ if and only if $dim(Var(r_1,\dots,r_h)) \leq 2n-h$ (see \cite[Proposition 1.1]{Q}). Hence, in order to prove the result it is enough to look at the dimension of $Var(J)$. In fact,
		\begin{align*}
			dim(Var(J))&=max_{W \subseteq V}\{dim(W_{\Omega}^{\perp})+dim(W_{\Omega})\}\\
			&=max_{W \subseteq V}\{dim(W_{\Omega}^{\perp}/(W_{\Omega}^{\perp} \cap {^{\perp}V}_{\Omega}))+ dim(W_{\Omega}^{\perp} \cap {^{\perp}V}_{\Omega})\\
			&+dim(W_{\Omega}/(W_{\Omega} \cap {^{\perp}V}_{\Omega}))+dim(W_{\Omega} \cap {^{\perp}V}_{\Omega})\}\\
			&\leq n-dim({^{\perp}V})+2dim({^{\perp}V})=2n-h
		\end{align*}
		since on $V_{\Omega}/{^{\perp}V}_{\Omega}$ the bilinear form is non-degenerate and $dim(^\perp V)=n-h$, therefore
		$$dim(W_{\Omega}/(W_{\Omega} \cap {^{\perp}V}_{\Omega}))+dim(W_{\Omega}^{\perp}/(W_{\Omega}^{\perp} \cap {^{\perp}V}_{\Omega}))=dim(V_{\Omega}/{^{\perp}V}_{\Omega})=n-dim({^{\perp}V}).$$
		
		Hence, the sequence is regular in $\Omega[x_1,\dots,x_n,y_1,\dots,y_n]$, and so in $\Z/2[x_1,\dots,x_n,y_1,\dots,y_n]$. 
	\end{proof}
	
	At this point, consider the standard embeddings $O_2^{\times m} \hookrightarrow O_{2m}$ and $O_2^{\times m} \times O_1 \hookrightarrow O_{2m+1}$, which induce respectively the ring homomorphisms compatible with Steenrod operations
	$$\alpha_{2m}:H(BO_{2m}) \cong H[u_1,\dots,u_{2m}] \rightarrow H(BO_2)^{\otimes m} \cong H[x_1,y_1,\dots,x_m,y_m]$$
	and 
	$$\alpha_{2m+1}:H(BO_{2m+1}) \cong H[u_1,\dots,u_{2m+1}] \rightarrow H(BO_2)^{\otimes m} \otimes_H H(BO_1) \cong H[x_1,y_1,\dots,x_m,y_m,x_{m+1}]$$
	where $x_i$ is in bidegree $(0)[1]$ and $y_i$ is in bidegree $(1)[2]$ for any $i$. By tensoring them with $\Z/2$ over $H$, one obtains the ring homomorphisms
	$$\beta_{2m}:\Z/2[u_1,\dots,u_{2m}] \rightarrow \Z/2[x_1,y_1,\dots,x_m,y_m]$$
	and
	$$\beta_{2m+1}:\Z/2[u_1,\dots,u_{2m+1}] \rightarrow \Z/2[x_1,y_1,\dots,x_m,y_m,x_{m+1}].$$
	
	Let $S_n$ be the polynomial ring $\Z/2[u_1,\dots,u_n]$ and $R_n$ be $\Z/2[x_1,y_1,\dots,x_m,y_m]$, if $n=2m$, and $\Z/2[x_1,y_1,\dots,x_m,y_m,x_{m+1}]$, if $n=2m+1$. There exist commutative diagrams
	
	$$
	\xymatrix{
		H(BO_{2m}) \ar@{->}[r]^(0.47){\alpha_{2m}} \ar@{->}[d]_{\gamma_{2m}} & H(BO_2)^{\otimes m} \ar@{->}[d]^{\delta_{2m}}\\
		S_{2m} \ar@{->}[r]_{\beta_{2m}} & R_{2m}
	}
	\hspace{1,5cm}
	\xymatrix{
		H(BO_{2m+1}) \ar@{->}[r]^(0.38){\alpha_{2m+1}} \ar@{->}[d]_{\gamma_{2m+1}} & H(BO_2)^{\otimes m} \otimes H(BO_1) \ar@{->}[d]^{\delta_{2m+1}}\\
		S_{2m+1} \ar@{->}[r]_{\beta_{2m+1}} & R_{2m+1}
	}
	$$
	where $\gamma_n$ and $\delta_n$ are the respective reduction maps along $H \rightarrow \Z/2$. By Whitney sum formula (see \cite[Proposition 3.1.13]{SV}) and since $\tau$ is killed in $R_{2m}$, we have that
	$$\beta_{2m}(u_{2j})=\sigma_j(y_1,\dots,y_m)$$ 
	and 
	$$\beta_{2m}(u_{2j+1})=\sum_{i=1}^m x_i\sigma_j(y_1,\dots,y_{i-1},y_{i+1},\dots,y_m)$$ 
	where $\sigma_j$ is the $j$-th elementary symmetric polynomial. Similar formulas hold in $R_{2m+1}$, i.e. we have that 
	$$\beta_{2m+1}(u_{2j})=\sigma_j(y_1,\dots,y_m)$$ 
	and 
	$$\beta_{2m+1}(u_{2j+1})=\sum_{i=1}^m x_i\sigma_j(y_1,\dots,y_{i-1},y_{i+1},\dots,y_m)+x_{m+1}\sigma_j(y_1,\dots,y_m).$$
	
	Before proceeding we need the following technical lemma on regular sequences.
	
	\begin{lem}\label{tech}
		Let $f:A=\Z/2[a_1,\dots,a_m] \rightarrow B=\Z/2[b_1,\dots,b_n]$ be a ring homomorphism, where $deg(b_i)=1$ for any $i$ and $f(a_j)$ is a homogeneous polynomial in $B$ of positive degree $\alpha_j$ for any $j$. Moreover, let $r_1,\dots,r_k$ be a sequence of elements of $A$. If $f(r_1),\dots,f(r_k)$ is a regular sequence in $B$, then $r_1,\dots,r_k$ is a regular sequence in $A$.
	\end{lem}
	\begin{proof}
		Let $C$ be the polynomial ring $\Z/2[b_1,\dots,b_n,c_1,\dots,c_m]$, with $deg(c_i)=1$ for any $i$. Define the ring homomorphisms $i:B \rightarrow C$, $h:C \rightarrow B$ and $g:A \rightarrow C$ by $i(b_j)=b_j$, $h(b_j)=b_j$ and $h(c_j)=0$, and $g(a_j)=if(a_j)+c_j^{\alpha_j}$. Note that $hi=id_B$, $f=hg$ and $g(a_j)$ is homogeneous in $C$ for any $j$. The sequence $b_1,\dots,b_n,g(a_1),\dots,g(a_m)$ is regular in $C$, so it is $g(a_1),\dots,g(a_m)$ since regular sequences of homogeneous elements of positive degree permute (see for example \cite[Corollary 17.2]{E}). From \cite[Proposition 1]{H} it follows that $C$ is a free $A$-module. At this point, note that $c_1,\dots,c_m,if(r_1),\dots,if(r_k)$ is a regular sequence in $C$ essentially by hypothesis. Hence, the sequence $g(r_1),\dots,g(r_k)$ is regular in $C$, since $g(r_j)+if(r_j) \in ker(h)=(c_1,\dots,c_m)$ for any $j$. The fact that $C$ is a free $A$-module via $g$ implies that $r_1,\dots,r_k$ is regular in $A$, which is what we aimed to show. 
	\end{proof}
	
	\begin{thm}\label{seq}
		The sequence $\gamma_n(u_1),\gamma_n(u_2),\gamma_n(Sq^1u_2),\dots,\gamma_n(Sq^{2^{h(n)-2}} Sq^{2^{h(n)-3}}\cdots Sq^1u_2)$ is regular in $S_n$, where $h(n)$ depends on $n$ as in the following table.
		\begin{center}
			\begin{tabular}{|c|c|}
				\hline
				\textbf{n} &\textbf{h(n)}\\
				\hline
				8l+1 &4l\\
				\hline
				8l+2 &4l+1\\
				\hline
				8l+3 &4l+1\\
				\hline
				8l+4 &4l+1\\
				\hline
				8l+5 &4l+2\\ 
				\hline
				8l+6 &4l+3\\ 
				\hline
				8l+7 &4l+3\\ 
				\hline
				8l+8 &4l+3\\ 
				\hline
			\end{tabular}
		\end{center}
	\end{thm}
	\begin{proof}
		By Lemma \ref{tech} we can check the regularity of the needed sequence by looking at its image under $\beta_n$. Indeed, we will show the regularity of the sequence
		$$\beta_n\gamma_n(u_1),\beta_n\gamma_n(u_2),\beta_n\gamma_n(Sq^1u_2),\dots,\beta_n\gamma_n(Sq^{2^{h(n)-2}} Sq^{2^{h(n)-3}}\cdots Sq^1u_2).$$
		
		First, consider the case $n=2m$. Then, $\beta_{2m}\gamma_{2m}(u_1)=\beta_{2m}(u_1)=\sum_{i=1}^m x_i$ and $\beta_{2m}\gamma_{2m}(u_2)=\beta_{2m}(u_2)=\sum_{i=1}^m y_i$. Moreover, since $\tau$ is killed in $R_{2m}$, we have that
		\begin{align*}
			\beta_{2m}\gamma_{2m}(Sq^{2^l}\cdots Sq^1u_2)&=\delta_{2m}\alpha_{2m}(Sq^{2^l}\cdots Sq^1u_2)=\delta_{2m}(Sq^{2^l}\cdots Sq^1\alpha_{2m}(u_2))\\
			&=\sum_{i=1}^m \delta_{2m}(Sq^{2^l}\cdots Sq^2Sq^1y_i)=\sum_{i=1}^m \delta_{2m}(Sq^{2^l}\cdots Sq^2(x_iy_i))=\sum_{i=1}^m x_iy_i^{2^l}.
		\end{align*}
		
		Modulo $\beta_{2m}\gamma_{2m}(u_1)$ and $\beta_{2m}\gamma_{2m}(u_2)$, $\beta_{2m}\gamma_{2m}(Sq^{2^l}\cdots Sq^1u_2)=B(x,y^{2^l})$, where $B$ is the bilinear form over an $m-1$-dimensional $\Z/2$-vector space $V$ defined by $B(x,y)=\sum_{i \neq j=1}^{m-1} x_iy_j$. Note that 
		$$dim({^{\perp}V})=
		\begin{cases}
		0, & if \: m \: is \: odd \\
		1, & if \: m \: is \: even
		\end{cases}.
		$$
		
		In fact, from $B(e_i,y)=0$ it follows that $y_1+\dots+y_{i-1}+y_{i+1}+ \dots +y_{m-1}=0$ for any $i \leq m-1$, where $e_i$ is the vector in $V$ which has coordinates which are all $0$ but the $i$-th that is a $1$. Hence, 
		$${^{\perp}V}=
		\begin{cases}
		0, & if \: m \: is \: odd \\
		\langle (1,\dots,1)\rangle  , & if \: m \: is \: even
		\end{cases}.
		$$
		
		Corollary \ref{reg} implies that the sequence
		$$\beta_{2m}\gamma_{2m}(u_1),\beta_{2m}\gamma_{2m}(u_2),\beta_{2m}\gamma_{2m}(Sq^1u_2),\dots,\beta_{2m}\gamma_{2m}(Sq^{2^{h(2m)-2}} Sq^{2^{h(2m)-3}}\cdots Sq^1u_2)$$ 
		is regular in $R_{2m}$ where 
		$$h(2m)=
		\begin{cases}
		m, & if \: m \: is \: odd \\
		m-1, & if \: m \: is \: even
		\end{cases}.
		$$
		
		Therefore, the sequence
		$$\gamma_{2m}(u_1),\gamma_{2m}(u_2),\gamma_{2m}(Sq^1u_2),\dots,\gamma_{2m}(Sq^{2^{h(2m)-2}} Sq^{2^{h(2m)-3}}\cdots Sq^1u_2)$$ 
		is regular in $S_{2m}$ for the same values of $h(2m)$.
		
		Now, consider the case $n=2m+1$. Similarly to the previous case, $\beta_{2m+1}\gamma_{2m+1}(u_1)=\beta_{2m+1}(u_1)=\sum_{i=1}^{m+1} x_i$ and $\beta_{2m+1}\gamma_{2m+1}(u_2)=\beta_{2m+1}(u_2)=\sum_{i=1}^m y_i$. Moreover, we have that
		\begin{align*}
			\beta_{2m+1}\gamma_{2m+1}(Sq^{2^l}\cdots Sq^1u_2)&=\delta_{2m+1}\alpha_{2m+1}(Sq^{2^l}\cdots Sq^1u_2)=\delta_{2m+1}(Sq^{2^l}\cdots Sq^1\alpha_{2m+1}(u_2))\\
			&=\sum_{i=1}^m \delta_{2m+1}(Sq^{2^l}\cdots Sq^2Sq^1y_i)=\sum_{i=1}^m \delta_{2m+1}(Sq^{2^l}\cdots Sq^2(x_iy_i))=\sum_{i=1}^m x_iy_i^{2^l}.
		\end{align*}
		
		Modulo $\beta_{2m+1}\gamma_{2m+1}(u_1)$ and $\beta_{2m+1}\gamma_{2m+1}(u_2)$, $\beta_{2m+1}\gamma_{2m+1}(Sq^{2^l}\cdots Sq^1u_2)=B(x,y^{2^l})$, where $B$ is the bilinear form over an $m$-dimensional $\Z/2$-vector space $V$ defined by $B(x,y)=\sum_{i=1}^{m-1} (x_i+x_m)y_i $. In this case, $dim({^{\perp}V})=1$. In fact, from $B(e_i,y)=0$ it follows that $y_i=0$ for any $i \leq m-1$. Hence, 
		${^{\perp}V}=\langle (0,\dots,0,1)\rangle  $, from which it follows by Corollary \ref{reg} that the sequence
		$$\beta_{2m+1}\gamma_{2m+1}(u_1),\beta_{2m+1}\gamma_{2m+1}(u_2),\beta_{2m+1}\gamma_{2m+1}(Sq^1u_2),\dots,\beta_{2m+1}\gamma_{2m+1}(Sq^{2^{h(2m+1)-2}} Sq^{2^{h(2m+1)-3}}\cdots Sq^1u_2)$$ 
		is regular in $R_{2m+1}$ where $h(2m+1)=m$. Therefore
		$$\gamma_{2m+1}(u_1),\gamma_{2m+1}(u_2),\gamma_{2m+1}(Sq^1u_2),\dots,\gamma_{2m+1}(Sq^{2^{h(2m+1)-2}} Sq^{2^{h(2m+1)-3}}\cdots Sq^1u_2)$$ 
		is a regular sequence in $S_{2m+1}$, where $h(2m+1)=m$. This completes the proof. 
	\end{proof}
	
	Define the elements $\theta_j$ in $H(BSO_n)$ inductively by the following formulas:
	$$\theta_0=u_2;$$
	$$\theta_{j+1}=Sq^{2^j}\theta_j.$$
	
	\begin{cor}\label{tauseq}
		The sequence $\tau,\theta_0,\dots,\theta_{h(n)-1}$ is regular in $H(BSO_n)$, where $h(n)$ depends on $n$ as in the table of Theorem \ref{seq}.
	\end{cor}
	\begin{proof}
		First, note that all $\theta_j$ are obtained from $u_2$ by using only Wu formula (Proposition \ref{Wu}) and Cartan formula where elements of $K^M(k)/2$ are never involved, from which it follows that every $\theta_j$ is an element of $\Z/2[\tau,u_2,\dots,u_n]$. Since $K^M(k)/2$ is free over $\Z/2$, it is enough to show the regularity of the sequence in $\Z/2[\tau,u_2,\dots,u_n]$. Then, the result follows from Theorem \ref{seq} by noticing that, modulo $\tau$ and $u_1$, $\theta_j=i_n\gamma_n(Sq^{2^{j-1}}\cdots Sq^1u_2)$, where $i_n$ is the inclusion of $S_n$ in $H(BO_n)$. 
	\end{proof}
	
	At this point, let us consider three homomorphisms $i:H_{top}(BSO_n) \rightarrow H(BSO_n)$, $h:H_{top}(BSO_n) \rightarrow H(BSO_n)$ and $t:H(BSO_n) \rightarrow H_{top}(BSO_n)$, where $i$ is defined by imposing $i(w_i)=u_i$ and extending to a ring homomorphism, $h$ by imposing, for any monomial $x$, $h(x)=\tau^{[{\frac{p_{i(x)}} {2}}-q_{i(x)}]}i(x)$, where $(q_{i(x)})[p_{i(x)}]$ is the bidegree of $i(x)$, and extending linearly and $t$ by imposing $t(u_i)=w_i$, $t(\tau)=1$ and $t(K_r^M(k)/2)=0$ for any $r>0$ and extending to a ring homomorphism.
	
	We start by describing some properties of these homomorphisms. First of all, $i$ and $h$ are graded with respect to the usual grading in $H_{top}(BSO_n)$ and the topological degree in $H(BSO_n)$. Besides, by the very definition of $h$, $h(x)$ has bidegree $([\frac{p_{i(x)}}{2}])[p_{i(x)}]$ for any homogeneous polynomial $x$. On the other hand, we notice that $h$ is not a ring homomorphism. Anyway, we have the following lemmas.
	
	\begin{lem}\label{h}
		For any homogeneous polynomials $x$ and $y$ in $H_{top}(BSO_n)$, we have that $h(xy)=\tau^{\epsilon}h(x)h(y)$, where $\epsilon$ is $1$ if $p_{i(x)}p_{i(y)}$ is odd and $0$ otherwise.
	\end{lem}
	\begin{proof}
		At first consider two monomials $x$ and $y$. Then, we get
		$$h(xy)=\tau^{[{\frac {p_{i(x)}+p_{i(y)}} 2}-q_{i(x)}-q_{i(y)}]}i(xy)=\tau^{{\epsilon}+[{\frac{p_{i(x)}} {2}}-q_{i(x)}]+[{\frac{p_{i(y)}} {2}}-q_{i(y)}]}i(x)i(y)=\tau^{\epsilon}h(x)h(y)$$ 
		where $\epsilon$ is $1$ if $p_{i(x)}p_{i(y)}$ is odd and $0$ otherwise. For homogeneous polynomials $x=\sum_{j=0}^l x_j$ and $y=\sum_{k=0}^m y_k$, where $x_j$ and $y_k$ are monomials, we have 
		$$h(xy)=h(\sum_{j=0}^l\sum_{k=0}^mx_jy_k)=\sum_{j=0}^l\sum_{k=0}^mh(x_jy_k)=\sum_{j=0}^l\sum_{k=0}^m\tau^{\epsilon_{jk}}h(x_j)h(y_k)$$
		where $\epsilon_{jk}$ is $1$ if $p_{i(x_j)}p_{i(y_k)}$ is odd and $0$ otherwise. Now, we recall that $p_{i(x_j)}=p_{i(x)}$ and $p_{i(y_k)}=p_{i(y)}$ for any $j$ and $k$, from which it immediately follows that $h(xy)=\tau^{\epsilon}\sum_{j=0}^l\sum_{k=0}^mh(x_j)h(y_k)=\tau^{\epsilon}h(x)h(y)$, where $\epsilon$ is $1$ if $p_{i(x)}p_{i(y)}$ is odd and $0$ otherwise. 
	\end{proof} 
	
	\begin{lem}\label{ht}
		For any homogeneous (respect to bidegree) $z \in \Z/2[\tau,u_2, \dots ,u_n]$, we have that $ht(z)=\tau^{[{\frac{p_z} {2}}-q_z]}z$ (where $[{\frac{p_z} {2}}-q_z]$ can possibly be negative).
	\end{lem}
	\begin{proof}
		Write $z$ as $\sum_{j=0}^m z_j$, where $z_j$ are monomials in $\Z/2[\tau,u_2, \dots ,u_n]$. Then,
		$$ht(z)=\sum_{j=0}^m ht(z_j)=\sum_{j=0}^m \tau^{[{\frac{p_{it(z_j)}} {2}}-q_{it(z_j)}]}it(z_j).$$
		
		Notice that $z_j=\tau^{n_j}x_j$, for some monomials $x_j$ in $\Z/2[u_2, \dots ,u_n]$. By the very definition of $i$ and $t$ we get that $it(z_j)=x_j$. Thus, 
		$$ht(z)=\sum_{j=0}^m \tau^{[{\frac{p_{x_j}} {2}}-q_{x_j}]}x_j=\sum_{j=0}^m \tau^{[{\frac{p_{x_j}} {2}}-q_{x_j}-n_j]}z_j=\tau^{[{\frac{p_z} {2}}-q_z]}z$$ since $p_{x_j}=p_{z_j}=p_z$ and $q_{x_j}+n_j=q_{z_j}=q_z$. 
	\end{proof}  
	
	\begin{lem}\label{th}
		For any $j$, $t(\theta_j)=\rho_j$ and $h(\rho_j)=\theta_j$.
	\end{lem}
	\begin{proof}
		Since a Wu formula (Proposition \ref{Wu}) holds even in the motivic case by \ref{Wu}, we get that $t(\theta_j)=\rho_j$ by the very definition of $t$. Then, $h(\rho_j)=ht(\theta_j)=\theta_j$ by Lemma \ref{ht} and by recalling that $\theta_j$ is in bidegree $(2^{j-1})[2^j+1]$. 
	\end{proof}
	
	At this point, denote by $I_j$ the ideal in $H(BSO_n)$ generated by $\theta_0, \dots ,\theta_{j-1}$ and by $I^{top}_j$ the ideal in $H_{top}(BSO_n)$ generated by $\rho_0, \dots ,\rho_{j-1}$. We are now ready to prove the main result of this section.
	
	\begin{thm}\label{MQ1}
		The sequence $\theta_0,\dots,\theta_{k(n)-1}$ is regular in $H(BSO_n)$. Moreover, $\theta_{k(n)} \in I_{k(n)}$, where $k(n)$ depends on $n$ as in the table of Theorem \ref{Q1}.
	\end{thm}
	\begin{proof}
		Since $K^M(k)/2$ is free over $\Z/2$ we just need to show the regularity of the needed sequence in $\Z/2[\tau,u_2, \dots ,u_n]$. From the fact that regular sequences of homogeneous elements of positive degree permute and by Corollary \ref{tauseq}, we immediately deduce the regularity of the sequence for $n=0,1,2,6 \: and \: 7(mod \: 8)$, since in these cases $h(n)=k(n)$. Now, suppose $n=3,4 \: or \: 5(mod \: 8)$. In these cases, $h(n)=k(n)-1$, therefore Corollary \ref{tauseq} implies that the sequence $\tau,\theta_0,\dots,\theta_{k(n)-2}$ is regular. Let $z$ be a homogeneous polynomial in $\Z/2[\tau,u_2,\dots,u_n]$ such that $z\theta_{k(n)-1} \in I_{k(n)-1}$. Then, we deduce that $t(z)\rho_{k(n)-1} \in I_{k(n)-1}^{top}$. It follows from Theorem \ref{Q1} that $t(z)=\sum_{l=0}^{k(n)-2} \psi_l\rho_l$ for some homogeneous $\psi_l \in H_{top}(BSO_n)$ and, after applying $h$, we obtain $\tau^{[{\frac{p_z} {2}}-q_z]}z=\sum_{l=0}^{k(n)-2} \tau^{\epsilon_l}h(\psi_l)\theta_l$ by Lemmas \ref{h}, \ref{ht} and \ref{th}. Hence, the regularity of $\theta_0,\dots,\theta_{k(n)-2},\tau$ implies that $z \in I_{k(n)-1}$ and we obtain the regularity of the sequence $\theta_0,\dots,\theta_{k(n)-1}$. At this point, we only need to show that $\theta_{k(n)} \in I_{k(n)}$. Note that $\rho_{k(n)} \in I_{k(n)}^{top}$ by Theorem \ref{Q2}. Hence, $\rho_{k(n)}=\sum_{l=0}^{k(n)-1} \phi_l\rho_l$ for some homogeneous $\phi_l \in H_{top}(BSO_n)$ and, after applying $h$, we obtain $\theta_{k(n)}=\sum_{l=0}^{k(n)-1} \tau^{\epsilon_l}h(\phi_l)\theta_l$ by Lemmas \ref{h} and \ref{th}. Thus, $\theta_{k(n)} \in I_{k(n)}$, which completes the proof. 
	\end{proof}
	
	\section{The motivic cohomology ring of $BSpin_n$}
	
	In this section we prove a motivic version of Theorem \ref{Q2}. The general strategy consists in using the grid of long exact sequences in motivic cohomology from Diagram \ref{d1} in Section \ref{D} in order to get the result by an inductive argument. This method allows us to lift, not only subtle classes, but even relations among them from the cohomology of $BSpin_n$ to the cohomology of $BSpin_{n+1}$. These relations are essentially the elements $\theta_j$ of the motivic regular sequences encountered in the previous section. Moreover, we see that a new subtle class $v_{2^{k(n)}}$ appears in the motivic cohomology of $BSpin_n$ and the obstruction to lift it to the cohomology of $BSpin_{n+1}$ is detected by the increasing of the lenght of the regular sequence moving from $n$ to $n+1$. In the proof of the main theorem it is essential to deal with the two possible cases separately: on the one hand the case that $v_{2^{k(n)}}$ is liftable and the lenght of the regular sequence stays unchanged, i.e. $k(n+1)=k(n)$, on the other the case that $v_{2^{k(n)}}$ is not liftable and the lenght of the regular sequence increases by one, i.e. $k(n+1)=k(n)+1$. Furthermore, we notice that when $v_{2^{k(n)}}$ is not liftable, then ``almost" its square is so, giving rise to a new extra class $v_{2^{k(n)+1}}$ in doubled degrees. 
	
	We start by showing that, as in topology, the second subtle Stiefel-Whitney class is trivial in the motivic cohomology ring $H(BSpin_n)$.
	
	\begin{lem}\label{u2}
		For any $n \geq 2$, $u_2$ is trivial in $H(BSpin_n)$. Moreover, there exists a unique element $x_0$ in $H(Cone(a_n))$ such that $b_n^*(x_0)=u_2$.
	\end{lem}
	\begin{proof}
		Recall that $SO_2 \cong Spin_2 \cong G_m$, where $G_m$ is the multiplicative group, and the morphism from $Spin_2$ to $SO_2$ is the double cover $G_m \xrightarrow{(\cdot)^2} G_m$, which induces the map on classifying spaces $a_2:BG_m \rightarrow BG_m$. By Kummer theory, the induced morphism on Picard groups $Pic(BG_m) \rightarrow Pic(BG_m)$ is multiplication by $2$. Now, recall that $ Pic(BG_m) \cong H^{2,1}(BG_m,\Z)$ (see \cite[Corollary 4.2]{MVW}). Then, for $n=2$ the homomorphism 
		$$a_2^*:H(BG_m) \cong H[u_2] \rightarrow H(BG_m) \cong H[v_2]$$ 
		sends $u_2$ to $2v_2$, hence $u_2=0$ in $H(BSpin_2)$.
		
		Now, suppose $u_2=0$ in $H(BSpin_n)$, then $u_2$ should be divisible by $u_{n+1}$ in $H(BSpin_{n+1})$, which forces $u_2$ to be trivial by degree reasons. Therefore, by induction, $u_2=0$ in $H(BSpin_n)$ for any $n$. It immediately follows that there exists $x_0$ in $H(Cone(a_n))$ such that $b_n^*(x_0)=u_2$ for any $n \geq 2$. We prove its uniqueness by showing that $b_n^*$ is a monomorphism in bidegree $(1)[2]$. First of all we notice that, for any $n \geq 2$, $H^{1,1}(BSpin_n) \cong K^M_1(k)/2$ by induction on $n$ and by observing that $\widetilde{g}^*$ is an isomorphism in bidegree $(1)[1]$. Hence, $c_n^*:H^{1,1}(BSpin_n) \rightarrow H^{2,1}(Cone(a_n))$ is the zero homomorphism, since the composition $H^{1,1} \rightarrow H^{1,1}(BSO_n) \rightarrow H^{1,1}(BSpin_n)$ is surjective and, therefore, so is the second map. It follows that $b_n^*:H^{2,1}(Cone(a_n)) \rightarrow H^{2,1}(BSO_n)$ is a monomorphism, as we aimed to show. 
	\end{proof}
	
	From the previous lemma, for any $n \geq 2$, we have a canonical set of elements $x_j$ in $H(Cone(a_n))$ defined by $x_j=Sq^{2^{j-1}} \cdots Sq^1x_0$ for any $j>0$. Denote by $\langle x_0, \dots ,x_{j-1}\rangle  $ the $H(BSO_n)$-submodule of $H(Cone(a_n))$ generated by $x_0, \dots ,x_{j-1}$. Before proceeding we need the following lemma.
	
	\begin{lem}\label{x}
		For any $j \geq 1$, $x_j \notin \langle x_0, \dots ,x_{j-1}\rangle$ in $H(Cone(a_2))$, and consequently in any $H(Cone(a_n))$.
	\end{lem}
	\begin{proof}
		We start by considering the Bockstein homomorphism $\beta$ associated to the short exact sequence $0 \rightarrow \Z \rightarrow \Z \rightarrow \Z/2 \rightarrow 0$. The homomorphism $a_2^*$ on cohomology with integer coefficients sends $u_2$ to $2v_2$ where $v_2$ is the generator of $H(BSpin_2) \cong H(BG_m)$ and so is injective, hence $b_2^*$ is the $0$ homomorphism on cohomology with integer coefficients, from which it follows that $x_0$ cannot come from any integral cohomology class. Thus, $y=\beta(x_0) \neq 0$. Moreover, since $u_2$ comes from an integral cohomology class, we have $b_2^*(y)=0$, so $y=mc_2^*(v_2)$ for some integer $m$. At this point we notice that $mv_2$ is in the image of $a_2^*$ for any even $m$, so $m$ must be odd, which implies that $y$ is not divisible by $2$, since $v_2 \: mod(2)$ is not in the image of $a_2^*$. This is enough to conclude that 
		$$x_1=Sq^1x_0=\beta(x_0) \: mod(2)=y \: mod(2) \neq 0.$$
		
		Hence, $x_1=c_2^*(v_2)$ from which we deduce that 
		$$x_j=Sq^{2^{j-1}} \dots Sq^2x_1=c_2^*(Sq^{2^{j-1}} \dots Sq^2v_2)=c_2^*(v_2^{2^{j-1}})$$
		by \cite[Lemma 9.8]{V2}, since $Sq^{2^{i-1}} \dots Sq^2v_2$ is in bidegree $(2^{i-1})[2^i]$ for any $i$. Now, suppose that $x_j \in \langle x_0, \dots ,x_{j-1}\rangle  $, in other words $x_j=\sum_{i=0}^{j-1}\phi_i x_i$ for some $\phi_i \in H[u_2]$. Then, we would have that 
		$$\phi_0u_2=b_2^*(x_j+\sum_{i=0}^{j-1}\phi_i x_i)=0$$ 
		which implies $\phi_0=0$. Moreover, since positive powers of $u_2$ act trivially on $H[v_2]$ (with $\Z/2$-coefficients), we have that
		$$c_2^*(v_2^{2^{j-1}})=c_2^*(v_2^{2^{j-1}}+\sum_{i=1}^{j-1}\phi_i v_2^{2^{i-1}})=0$$
		that is impossible since $c_2^*$ is injective on the slope $2$ line (above zero), which comes from the fact that $H(BSO_2) \cong H[u_2]$ and $a_2^*(u_2)=0$. 
	\end{proof}
	
	At this point, we are ready to prove our main result which provides the complete description of the motivic cohomology of $BSpin_n$ over fields of characteristic different from $2$ containing $\sqrt{-1}$.
	
	\begin{thm}\label{MQ2}
		For any $n \geq 2$, there exists a cohomology class $v_{2^{k(n)}}$ of bidegree $(2^{k(n)-1})[2^{k(n)}]$ such that the natural homomorphism of $H$-algebras 
		$$H(BSO_n)/I_{k(n)} \otimes_H H[v_{2^{k(n)}}] \rightarrow H(BSpin_n)$$
		is an isomorphism, where $I_{k(n)}$ is the ideal generated by $\theta_0, \dots ,\theta_{k(n)-1}$ and $k(n)$ depends on $n$ as in the table of Theorem \ref{Q1}.
	\end{thm}
	\begin{proof}
		Our proof goes by induction on $n$, starting from $n=2$.
		
		\textit{Base case}: For $n=2$, $H(BSpin_2) \cong H(BG_m) \cong H[v_2]$ provides our induction base.
		
		\textit{Inductive step}: We denote by $\theta'_j$ and $\theta_j$ the class $Sq^{2^{j-1}} \cdots Sq^1u_2$ in $H(BSO_n)$ and $H(BSO_{n+1})$ respectively, by $I'_{k(n)}$ the ideal generated by the elements $u_2,\theta'_1, \dots ,\theta'_{{k(n)}-1}$, by $I_{k(n)}$ the ideal generated by $u_2,\theta_1, \dots ,\theta_{{k(n)}-1}$, by $x'_0$ and $x_0$ the unique lifts of $u_2$ to $H(Cone(a_n))$ and $H(Cone(a_{n+1}))$ respectively, by $x'_j$ the class $Sq^{2^{j-1}} \cdots Sq^1x'_0$ and by $x_j$ the class $Sq^{2^{j-1}} \cdots Sq^1x_0$.
		
		Now, suppose by induction hypothesis that we have an isomorphism
		\begin{equation}\label{ih}
		H(BSO_n)/I'_{k(n)} \otimes_H H[v_{2^{k(n)}}] \rightarrow H(BSpin_n)
		\end{equation} 
		where $k(n)$ is the value prescribed by the table of Theorem \ref{Q1}.
		
		Looking at the long exact sequence
		\begin{equation}\label{d2}
			\dots \rightarrow H^{*-1,*'}(BSpin_n) \rightarrow H^{*-n-1,*'-[(n+1)/2]}(BSpin_{n+1}) \xrightarrow{\text{$\cdot u_{n+1}$}} H^{*,*'}(BSpin_{n+1}) \rightarrow H^{*,*'}(BSpin_n) \rightarrow \dots
			\end{equation}
		from Diagram \ref{d1} in Section \ref{D} and by induction on degree we know that, in square degrees less than $2^{k(n)}$, in $H(BSpin_{n+1})$ there are only subtle Stiefel-Whitney classes, i.e. the homomorphism $a_{n+1}^*:H(BSO_{n+1}) \rightarrow H(BSpin_{n+1})$ is surjective in these degrees. Let $w$ be a class in $H^{2^{k(n)}-n,2^{{k(n)}-1}-[\frac{n+1}{2}]}(BSO_{n+1})$ such that $a_{n+1}^*(w)\widetilde{\alpha}=\widetilde{h}^*(v_{2^{k(n)}})$, where $\widetilde{\alpha}$ is the Thom class of the morphism $\widetilde{g}$. We point out that $$a_{n+1}^*(u_{n+1}w)=u_{n+1}a_{n+1}^*(w)=\widetilde{f}^*\widetilde{h}^*(v_{2^{k(n)}})=0.$$
		
		The following result, whose proof is reported at the end of this section, enables to complete the induction step. It is indeed the core proposition that permits to conduct the proof of our main theorem. 
		
		\begin{prop}\label{vp}
			Suppose we have a commutative diagram
			$$
			\xymatrix{
				H(BSO_{n+1}) \otimes_H H[v] \ar@{->}[r]^{g^* \otimes l} \ar@{->}[d]_{p_{n+1}} & H(BSO_n) \otimes_H H[c] \ar@{->}[d]^{p_n}\\
				H(BSpin_{n+1}) \ar@{->}[r]^{\widetilde{g}^*} & H(BSpin_n)
			}
			$$
			such that $v$ is a lift from $H(BSpin_n)$ to $H(BSpin_{n+1})$ of a monic homogeneous polynomial $c$ in $v_{2^{k(n)}}$ with coefficients in $H(BSO_n)$, $l(v)=c$ and $\widetilde{h}^*(c)=0$. 
			
			If $Im(\widetilde{h}^*)=Im(p_{n+1}) \cdot \widetilde{h}^*(v_{2^{k(n)}})$, then $ker(p_{n+1})=J_{k(n)}+(u_{n+1}w)$, where $J_{k(n)}$ is $I_{k(n)} \otimes_H H[v]$.
			
			If moreover $ker(\widetilde{h}^*)=Im(\widetilde{g}^*p_{n+1})$, then we get an isomorphism
			$$H(BSO_{n+1})/(I_{k(n)}+(u_{n+1}w)) \otimes_H H[v] \rightarrow H(BSpin_{n+1}).$$
		\end{prop}
		
        So, in order to finalize the proof we only need to find a cohomology class $v$ which satisfies the requirements of Proposition \ref{vp}. There are two possible cases: 1) $\widetilde{h}^*(v_{2^{k(n)}})=0$; 2) $\widetilde{h}^*(v_{2^{k(n)}}) \neq 0$.
		
		\textit{Case 1}: In this case $v_{2^{k(n)}}$ can be lifted to $H(BSpin_{n+1})$ so $w=0$ and we can choose $c=v_{2^{k(n)}}$. It follows that $Im(\widetilde{h}^*)=0=Im(p_{n+1}) \cdot \widetilde{h}^*(v_{2^{k(n)}})$ and $ker(\widetilde{h}^*)=H(BSpin_n)=Im(p_n)=Im(p_n(g^* \otimes l))=Im(\widetilde{g}^*p_{n+1})$, since in this case $p_n$ and $g^* \otimes l$ are surjective. So, by Proposition \ref{vp}, we have that the homomorphism
		$$H(BSO_{n+1})/I_{k(n)} \otimes_H H[v_{2^{k(n)}}] \rightarrow H(BSpin_{n+1})$$
		is an isomorphism. Furthermore, we observe that $k(n+1)=k(n)$ is the value predicted by the table of Theorem \ref{Q1} since $\theta_{k(n)} \in I_{k(n)}$ as it is zero in $H(BSpin_{n+1})$ (because $u_2$ is). This completes the first case.
		
		\textit{Case 2}: In this case we notice that the element $w$ such that $a_{n+1}^*(w)\widetilde{\alpha}=\widetilde{h}^*(v_{2^{k(n)}})$ must be different from $0$. 
		
		\begin{rem}\label{rem2}
			\normalfont
		Since $H(BSpin_n)$ is generated by $v_{2^{k(n)}}^i$ as a $H(BSO_n)$-module (and, so, as a $H(BSO_{n+1})$-module) by induction hypothesis, we have that $Im(\widetilde{h}^*)$ is generated by $\widetilde{h}^*(v_{2^{k(n)}}^i)$ as a $H(BSO_{n+1})$-module.
		\end{rem} 
	
	At this point, we need the following lemmas whose proofs are reported at the end of this section.
		
		\begin{lem}\label{Sqw}
			For any $m$, we have $Sq^ma_{n+1}^*(w) \in \langle a_{n+1}^*(w)\rangle  $, where $\langle a_{n+1}^*(w)\rangle  $ is the $H(BSO_{n+1})$-submodule of $H(BSpin_{n+1})$ generated by $a_{n+1}^*(w)$.
		\end{lem}
		
		\begin{lem}\label{lambdamu}
			For any $m>1$ there exist elements $\lambda_m$ and $\mu_m$ in $H(BSpin_{n+1})$ such that $\widetilde{h}^*(v_{2^{k(n)}}^m)=\lambda_m\widetilde{h}^*(v_{2^{k(n)}})$, $\widetilde{g}^*(\mu_m)=v_{2^{k(n)}}^m+\widetilde{g}^*(\lambda_m)v_{2^{k(n)}}$, $\lambda_m$ and $\mu_m$ are in the image of $H(BSO_{n+1}) \otimes_H H[\mu_2]$ and $\mu_m$ is divisible by $\mu_2$.
		\end{lem}
	
		Now consider the following commutative diagram
		$$
		\xymatrix{
			H(BSO_{n+1}) \otimes_H H[\mu_2] \ar@{->}[r]^(0.41){g^* \otimes l} \ar@{->}[d]_{p_{n+1}} & H(BSO_n) \otimes_H H[v_{2^{k(n)}}^2+\widetilde{g}^*(\lambda_2)v_{2^{k(n)}}] \ar@{->}[d]^{p_n}\\
			H(BSpin_{n+1}) \ar@{->}[r]^{\widetilde{g}^*} & H(BSpin_n).
		}
		$$
		
		From Lemma \ref{lambdamu} and from Remark \ref{rem2} we get that $Im(\widetilde{h}^*)=Im(p_{n+1}) \cdot \widetilde{h}^*(v_{2^{k(n)}})$. Then, by Proposition \ref{vp}, we obtain that $ker(p_{n+1})=J_{k(n)}+(u_{n+1}w)$.
		
		Note that, by looking at the long exact sequence \ref{d2} at the beginning of the proof and by induction on degree, $H^{2^{k(n)},2^{{k(n)}-1}}(BSpin_{n+1})$ consists only of subtle Stiefel-Whitney classes, since we are studying the case that $v_{2^{k(n)}}$ is not covered by $\widetilde{g}^*$ and so $a_{n+1}^*:H(BSO_{n+1}) \rightarrow H(BSpin_{n+1})$ is surjective also in bidegree $(2^{{k(n)}-1})[2^{k(n)}]$. Hence, 
		$$c_{n+1}^*:H^{2^{k(n)},2^{{k(n)}-1}}(BSpin_{n+1}) \rightarrow H^{2^{k(n)}+1,2^{{k(n)}-1}}(Cone(a_{n+1}))$$ 
		is the zero homomorphism and $b_{n+1}^*$ is injective in bidegree of $x_{k(n)}$, from which we deduce that $\theta_{k(n)} \notin I_{k(n)}$ since $x_{k(n)} \notin \langle x_0, \dots ,x_{{k(n)}-1}\rangle  $ by Lemma \ref{x}. Therefore, by observing that $ker(p_{n+1})=J_{k(n)}+(u_{n+1}w)$ and $p_{n+1}(\theta_{k(n)})=0$ we get that $\theta_{k(n)}+u_{n+1}w \in I_{k(n)}$ which implies that $ker(p_{n+1})=J_{{k(n)}+1}$. 
		
		In order to finish, we need the following lemma whose proof is reported at the end of this section.
		
		\begin{lem}\label{hg}
			The following identification holds in $H(BSpin_n)$: 
			$$ker(\widetilde{h}^*)=Im(\widetilde{g}^*p_{n+1}).$$
		\end{lem}
		
		Denote by $v_{2^{{k(n)}+1}}$ the class $\mu_2$, then by Proposition \ref{vp} we get that the homomorphism
		$$H(BSO_{n+1})/I_{{k(n)}+1} \otimes_H H[v_{2^{{k(n)}+1}}] \rightarrow H(BSpin_{n+1})$$
		is an isomorphism. Moreover, since $\theta_{k(n)} \notin I_{k(n)}$ we have that $\rho_{k(n)} \notin I^{top}_{k(n)}$ from which it follows that $k(n+1)=k(n)+1$ by Remark \ref{rem1}. This completes the proof of the second case.  
	\end{proof}

We conclude this section by providing the proofs of Proposition \ref{vp} and Lemmas \ref{Sqw}, \ref{lambdamu} and \ref{hg}. We remind the reader that in all the following proofs there is a running inductive assumption (see \ref{ih} at the beginning of the proof of Theorem \ref{MQ2}).

\begin{proof}[Proof of Proposition \ref{vp}]
	We want to prove that $p_{n+1}(x)=0$ implies $x \in J_{k(n)}+(u_{n+1}w)$ for any $x$. We proceed by induction on the square degree of $x$. The induction base is guaranteed by the fact that the degree $2$ part of the kernel is generated by $u_2$ and $u_2 \in I_{k(n)}$. Now, suppose that the claim is true for square degrees less than the square degree of $x$. We can write $x$ as $\sum_{j=0}^m \phi_jv^j$ for some $\phi_j \in H(BSO_{n+1})$. Notice that $p_n(g^* \otimes l)(x)=\widetilde{g}^*p_{n+1}(x)=0$, therefore $\sum_{j=0}^m p_ng^*(\phi_j)c^j=0$. From this we deduce that $p_ng^*(\phi_j)=0$ for any $j$ since by hypothesis $c$ is a monic polynomial in $v_{2^{k(n)}}$ in $H(BSpin_n)$, so $g^*(\phi_j) \in I'_{k(n)}$. Then, $\phi_j \in I_{k(n)}+(u_{n+1})$ since $\phi_j+ig^*(\phi_j) \in (u_{n+1})$ and $i(I'_{k(n)}) \subset I_{k(n)}+(u_{n+1})$, where $i$ is the inclusion of $H(BSO_n)$ in $H(BSO_{n+1})$ sending $u_l$ to $u_l$. Hence, there are $\psi_j \in H(BSO_{n+1})$ such that $\phi_j+u_{n+1}\psi_j \in I_{k(n)}$, from which it follows that $x+u_{n+1}z \in J_{k(n)}$ where $z=\sum_{j=0}^m \psi_jv^j$. Hence, $u_{n+1}p_{n+1}(z)=0$ which implies that 
	$$p_{n+1}(z)\widetilde{\alpha} \in Im(\widetilde{h}^*)=Im(p_{n+1}) \cdot \widetilde{h}^*(v_{2^{k(n)}})=Im(p_{n+1}) \cdot p_{n+1}(w){\widetilde{\alpha}}$$
	from which we deduce that there exists an element $y$ in $H(BSO_{n+1}) \otimes_H H[v]$ such that $p_{n+1}(z)=p_{n+1}(yw)$. Therefore, $z+yw \in J_{k(n)}+(u_{n+1}w)$ by induction hypothesis. It follows that $z \in J_{k(n)}+(w)$ and $x \in J_{k(n)}+(u_{n+1}w)$.
	
	In order to prove the last part of the proposition we show by induction on degree that, if $ker(\widetilde{h}^*)=Im(\widetilde{g}^*p_{n+1})$, then $p_{n+1}$ is surjective. The induction basis comes from the fact that, in square degree $\leq 2$, $H(BSpin_{n+1})$ is the same as the cohomology of the point. Take an element $x$ and suppose that $p_{n+1}$ is surjective in square degrees less than the square degree of $x$. From $\widetilde{g}^*(x) \in ker(\widetilde{h}^*)=Im(\widetilde{g}^*p_{n+1})$ it follows that there is an element $\chi$ in $H(BSO_{n+1}) \otimes_H H[v]$ such that $\widetilde{g}^*(x)=\widetilde{g}^*p_{n+1}(\chi)$. Therefore, $x+p_{n+1}(\chi)=u_{n+1}z$ for some $z \in H(BSpin_{n+1})$. By induction hypothesis $z=p_{n+1}(\zeta)$ for some element $\zeta \in H(BSO_{n+1}) \otimes_H H[v]$, hence $x=p_{n+1}(\chi+u_{n+1}\zeta)$, which is what we aimed to show. 
\end{proof}

\begin{proof}[Proof of Lemma \ref{Sqw}]
	We proceed by induction on $m$. For $m=0$ there is nothing to prove and for $m>2^{k(n)}-n$ we have that $Sq^mw=0$ by Corollary \ref{corWu}. Suppose the statement is true for integers less than $m \leq 2^{k(n)}-n$. Then, 
	$$Sq^m(u_{n+1}w)=\sum_{j=0}^m \tau^{j \: mod2}Sq^ju_{n+1}Sq^{m-j}w=\sum_{j=0}^m \tau^{j \: mod2}u_ju_{n+1}Sq^{m-j}w$$
	from which it follows, by applying $a_{n+1}^*$ and by noting that $u_{n+1}a_{n+1}^*(w)=0$, that $$0=Sq^m(u_{n+1}a_{n+1}^*(w))=\sum_{j=0}^m \tau^{j \: mod2}u_ju_{n+1}Sq^{m-j}a_{n+1}^*(w)=u_{n+1}Sq^ma_{n+1}^*(w)$$
	where all the elements but one in the sum disappear since by induction (on $m$) hypothesis $Sq^{m-j}a_{n+1}^*(w) \in \langle a_{n+1}^*(w)\rangle  $ for $j>0$ and $u_{n+1}a_{n+1}^*(w)=0$.
	
	Hence, $\widetilde{f}^*(Sq^ma_{n+1}^*(w)\widetilde{\alpha})=0$, from which it follows that $Sq^ma_{n+1}^*(w)\widetilde{\alpha} \in Im(\widetilde{h}^*)$. By Remark \ref{rem2}, we obtain that $Sq^ma_{n+1}^*(w)\widetilde{\alpha}=\sum_{i \geq 1} \phi_i \widetilde{h}^*(v_{2^{k(n)}}^i)$ for some $\phi_i \in H(BSO_{n+1})$. But, for any $i>1$, the square degree of $\widetilde{h}^*(v_{2^{k(n)}}^i)$ is greater than that of $Sq^ma_{n+1}^*(w)\widetilde{\alpha}$. We deduce that $Sq^ma_{n+1}^*(w)\widetilde{\alpha}=\phi_1 \widetilde{h}^*(v_{2^{k(n)}})$, from which it follows that
	$$Sq^ma_{n+1}^*(w)=\phi_1 a_{n+1}^*(w) \in \langle a_{n+1}^*(w)\rangle  $$
	which is what we aimed to prove. 
\end{proof}

	\begin{proof}[Proof of Lemma \ref{lambdamu}]
	We notice that, by Proposition \ref{SteenrodThom} and Corollary \ref{corWu},
	\begin{align*}
		\widetilde{h}^*(v_{2^{k(n)}}^2)&=\widetilde{h}^*(Sq^{2^{k(n)}}v_{2^{k(n)}})=Sq^{2^{k(n)}}(a_{n+1}^*(w)\widetilde{\alpha})\\
		&=(\tau^{n \: mod 2}Sq^{2^{k(n)}-n}a_{n+1}^*(w)u_n+\tau^{(n+1) \: mod 2}Sq^{2^{k(n)}-n-1}a_{n+1}^*(w)u_{n+1})\widetilde{\alpha}\\
		&=\tau^{n \: mod 2}Sq^{2^{k(n)}-n}a_{n+1}^*(w)u_n\widetilde{\alpha}
	\end{align*}
	since, by Lemma \ref{Sqw}, $Sq^{2^{k(n)}-n-1}a_{n+1}^*(w)\in \langle a_{n+1}^*(w)\rangle $ and $u_{n+1}a_{n+1}^*(w)=0$. Now, note that Lemma \ref{Sqw} also implies that $Sq^{2^{k(n)}-n}a_{n+1}^*(w)=ra_{n+1}^*(w)$ for some $r \in H(BSO_{n+1})$ which allows us to define the element $\lambda_2$ in $H(BSO_{n+1})$ as $\lambda_2=\tau^{n \: mod2}ru_n$. Then, we immediately obtain that $\widetilde{h}^*(v_{2^{k(n)}}^2)=\lambda_2\widetilde{h}^*(v_{2^{k(n)}})$. 
	
	Denote by $\mu_2$ a lift of $v_{2^{k(n)}}^2+\widetilde{g}^*(\lambda_2)v_{2^{k(n)}}$ to $H(BSpin_{n+1})$. Suppose the statement is true for $m$, so, taking into account that $\widetilde{h}^*$ is $H(BSpin_{n+1})$-linear, we have 
	$$\widetilde{h}^*(v_{2^{k(n)}}^{m+1})=\widetilde{h}^*((v_{2^{k(n)}}^m+\widetilde{g}^*(\lambda_m)v_{2^{k(n)}})v_{2^{k(n)}}+\widetilde{g}^*(\lambda_m)v_{2^{k(n)}}^2)=\mu_m\widetilde{h}^*(v_{2^{k(n)}})+\lambda_m\lambda_2\widetilde{h}^*(v_{2^{k(n)}}).$$
	
	Denote by $\lambda_{m+1}$ the element $\mu_m+\lambda_m\lambda_2$ and by $\mu_{m+1}$ the element $\lambda_m\mu_2$. Then, 
	$$\widetilde{g}^*(\mu_{m+1})=\widetilde{g}^*(\lambda_m\mu_2)=\widetilde{g}^*(\lambda_m)(v_{2^{k(n)}}^2+\widetilde{g}^*(\lambda_2)v_{2^{k(n)}})=\widetilde{g}^*(\lambda_m)v_{2^{k(n)}}^2+\widetilde{g}^*(\lambda_{m+1}+\mu_m)v_{2^{k(n)}}=$$
	$$(\widetilde{g}^*(\lambda_m)v_{2^{k(n)}}+\widetilde{g}^*(\lambda_{m+1})+v_{2^{k(n)}}^m+\widetilde{g}^*(\lambda_m)v_{2^{k(n)}})v_{2^{k(n)}}=v_{2^{k(n)}}^{m+1}+\widetilde{g}^*(\lambda_{m+1})v_{2^{k(n)}}$$
	and the proof is complete. 
\end{proof}

\begin{proof}[Proof of Lemma \ref{hg}]
	Let us set $\mu_1=\lambda_0=0$ and $\mu_0=\lambda_1=1$. Let $x$ be an element of the kernel of $\widetilde{h}^*$. We can write $x$ as $\sum_{j=0}^m\gamma_jv_{2^{k(n)}}^j$ with $\gamma_j \in H(BSO_{n+1})$. Then, by Lemma \ref{lambdamu}, 
	$$x=\sum_{j=0}^m\gamma_j(\widetilde{g}^*(\mu_j)+\widetilde{g}^*(\lambda_j)v_{2^{k(n)}})$$
	from which it follows by applying $\widetilde{h}^*$ that $\sum_{j=0}^m\gamma_j\lambda_j\widetilde{h}^*(v_{2^{k(n)}})=0$. Denote by $\sigma$ the element $\sum_{j=0}^m\gamma_j\lambda_j$ in $H(BSO_{n+1}) \otimes_H H[\mu_2]$. From 
	$$p_{n+1}(\sigma w)\widetilde{\alpha}=p_{n+1}(\sigma) a_{n+1}^*(w)\widetilde{\alpha}=p_{n+1}(\sigma) \widetilde{h}^*(v_{2^{k(n)}})=0$$ 
	we get $\sigma w \in J_{{k(n)}+1}$, since $ker(p_{n+1})=J_{{k(n)}+1}$. Thus, $\sigma w=\sum_{j=0}^{k(n)}\sigma_j\theta_j$ for some $\sigma_j \in H(BSO_{n+1}) \otimes_H H[\mu_2]$ and, multiplying by $u_{n+1}$, we obtain that $u_{n+1}\sigma w+u_{n+1}\sigma_{k(n)} \theta_{k(n)} \in J_{k(n)}$. On the other hand, $\theta_{k(n)}+u_{n+1}w \in I_{k(n)}$, from which it follows by multiplying by $\sigma$ that $\sigma \theta_{k(n)}+u_{n+1}\sigma w \in J_{k(n)}$. Hence, $(\sigma+u_{n+1}\sigma_{k(n)})\theta_{k(n)} \in J_{k(n)}$. By Theorem \ref{MQ1} we deduce that $\sigma+u_{n+1}\sigma_{k(n)} \in J_{k(n)}$, from which it follows that $\sigma \in J_{k(n)}+(u_{n+1})$. Therefore, $\widetilde{g}^*p_{n+1}(\sigma)=0$ in $H(BSpin_n)$ and
	$$x=\sum_{j=0}^m\gamma_j\widetilde{g}^*(\mu_j) \in Im(\widetilde{g}^*p_{n+1})$$
	as we aimed to show. 
\end{proof}

\section{The motivic cohomology ring of $BG_2$}
	
	In this section, we use our main result to compute the motivic cohomology ring of the Nisnevich classifying space of $G_2$. This enables us to obtain motivic invariants for $G_2$-torsors, i.e. octonion algebras.
	
	We start by noticing that there is a fiber sequence
	$$A_{q_8} \rightarrow BG_2 \rightarrow BSpin_7$$
	(see \cite[Proposition 3.1.1]{AHW}). We can exploit this sequence and previous results to compute the motivic cohomology ring of $BG_2$. Before proceeding, note that by Theorem \ref{MQ2} we know the complete description of $H(BSpin_7)$.
	
	\begin{thm}\label{g2}
		The motivic cohomology ring of $BG_2$ is completely described by
		$$H(BG_2) \cong H[u_4,u_6,u_7].$$
	\end{thm}
	\begin{proof}
		By applying Proposition \ref{Thom1} to the coherent morphism $\widehat{B}G_2 \rightarrow BSpin_7$ whose fiber is isomorphic to $A_{q_8}$ we get a Gysin long exact sequence of $H(BSpin_7)$-modules in motivic cohomology
		$$\dots \rightarrow H^{p-8,q-4}(BSpin_7) \rightarrow H^{p,q}(BSpin_7) \rightarrow H^{p,q}(BG_2)\rightarrow H^{p-7,q-4}(BSpin_7) \rightarrow \dots.$$
		
		Hence, in order to be able to describe $H(BG_2)$ we need only to understand where $1$ is sent under the morphism $H^{p-8,q-4}(BSpin_7) \rightarrow H^{p,q}(BSpin_7)$. Recall that from Theorem \ref{MQ2} we have that $H(BSpin_7) \cong H[u_4,u_6,u_7,v_8]$.
		
		Note that there is a commutative diagram
		$$
		\xymatrix{
			BSL_2 \ar@{->}[r]^(0.35){\Delta} \ar@{<->}[d]_{\cong} & BSL_2 \times BSL_2  \ar@{->}[r] \ar@{<->}[d]_{\cong} & BSL_4 \ar@{<->}[d]^{\cong}\\
			BSpin_3 \ar@{->}[r] & BSpin_4 \ar@{->}[r] & BSpin_6
		}
		$$
		where all the vertical maps are induced by the sporadic isomorphisms $SL_2 \cong Spin_3$, $SL_2 \times SL_2 \cong Spin_4$ and $SL_4 \cong Spin_6$. Recall that $H(BSL_n) \cong H[c_2,\dots,c_n]$ where $c_i$ is the Chern class in bidegree $(i)[2i]$ (see \cite[Proposition 3.2.7]{SV} which works in the same way for $BSL_n$). Then, we get a commutative diagram of motivic cohomology rings
		$$
		\xymatrix{
			H[c] \ar@{<-}[r]^(0.5){\Delta^*} \ar@{<->}[d]_{\cong} & H[c',c'']  \ar@{<-}[r] \ar@{<->}[d]_{\cong} & H[c_2,c_3,c_4] \ar@{<->}[d]^{\cong}\\
			H[v_4] \ar@{<-}[r] & H[u_4,v_4] \ar@{<-}[r] & H[u_4,u_6,v_8]
		}
		$$
		where the first vertical arrow identifies $c$ with $v_4$, the last vertical arrow identifies $c_2$ with $u_4$ and $c_3$ with $u_6$, $c'$ and $c''$ are sent both to $c$ and $c_4$ maps to $c'c''$. Now, note that $H(BSpin_6) \rightarrow H(BSpin_4)$ factors through $H(BSpin_5) \cong H[u_4,v_8]$. Since $\widetilde{h}^*:H(BSpin_4) \rightarrow H(BSpin_5)$ is nontrivial, the class $w$ defined just before Proposition \ref{vp} in the proof of Theorem \ref{MQ2} is equal to $1$ and, so, by Lemma \ref{lambdamu} we know that $\lambda_2=u_4$. Hence, $v_8$ maps to $v_4^2+u_4v_4$. Moreover, since $c$ is identified with $v_4$ and both $c'$ and $c''$ map to $c$, the second vertical arrow identifies $c'$ and $c''$ with $v_4$ and $v_4+u_4$. It follows that $c'c''$ is identified with $v_4^2+u_4v_4$. Therefore, $c_4$ is identified with $v_8$ since they are the only classes in their degrees that restrict to the same element.
		
		Moreover, we can notice that there is a cartesian square of simplicial schemes given by
		$$
		\xymatrix{
			BSL_3 \ar@{->}[r] \ar@{->}[d] & BSL_4 \cong BSpin_6 \ar@{->}[d]\\
			BG_2 \ar@{->}[r]& BSpin_7.
		}
		$$
		
		Recall that $H(BSL_3) \cong H[c_2,c_3]$, $H(BSL_4) \cong H[c_2,c_3,c_4]$ and $H(BSpin_6) \cong H[u_4,u_6,v_8]$ with the identifications $c_2=u_4$, $c_3=u_6$ and $c_4=v_8$ discussed above. Hence, by Corollary \ref{Thom3} we easily deduce that the morphism $H^{p-8,q-4}(BSpin_7) \rightarrow H^{p,q}(BSpin_7)$ sends $1$ to an element which maps to $c_4$ via the morphism $H(BSpin_7) \rightarrow H(BSL_4)$. Therefore, $H^{p-8,q-4}(BSpin_7) \rightarrow H^{p,q}(BSpin_7)$ can only be multiplication by $v_8+\{a\}u_7$, from which it immediately follows that $H(BG_2) \cong H[u_4,u_6,u_7]$, which is what we aimed to prove. 
	\end{proof}
	
	\section{Relations among subtle classes for $Spin_n$-torsors}
	
	In this section we deduce, just from the triviality of $u_2$ in the motivic cohomology of $BSpin_n$, some very simple relations among subtle classes in the motivic cohomology of the \v{C}ech simplicial scheme associated to a $Spin_n$-torsor. This provides information about the kernel invariant (see \cite[2.7.1]{SV}) of quadratic forms from $I^3$. 
	
	We start by recalling that there exists a map from $Spin_n$-torsors over the point to $n$-dimensional quadratic forms from $I^3$ which is surjective and has trivial kernel, where $I$ is the fundamental ideal in the Witt ring. Moreover, we have the following commutative diagram
	$$
	\xymatrix{
		\check{C}(X_q) \ar@{->}[r] \ar@{->}[d] & BSpin_n \ar@{->}[r] \ar@{->}[d] & BSO_n \ar@{->}[d]\\
		Spec(k) \ar@{->}[r]^(0.45){q} & B_{\acute{e}t}Spin_n \ar@{->}[r] & B_{\acute{e}t}SO_n
	}
	$$
	for any $n$-dimensional $q \in I^3$ and all above-diagonal classes in $H(BSpin_n)$ coming from the \'etale classifying space trivialise in $H(\check{C}(X_q))$, since the above-diagonal cohomology of a point is zero. Here $\check{C}(X_q)$ is the \v{C}ech simplicial scheme associated to the torsor $X_q=Iso\{q \leftrightarrow q_n\}$. In particular Chern classes $c_i(q)=\tau^{i \: mod2}u_i(q)^2$ are zero, as these are coming from the \'etale space (see \cite{SV} just before Thorem 3.1.1).
	
	From previous remarks we obtain the following proposition, which provides us with relations among subtle characteristic classes for quadratic forms from $I^3$.
	
	\begin{prop}\label{formula}
		For any $n$-dimensional $q \in I^3$, the following relations hold in $H(\check{C}(X_q))$ $$\sum_{h=0}^{2^j}u_{2^j-h}(q)u_{2^j+1+h}(q)=0$$ for any $j$ satisfying $2^j+1 \leq n$.
	\end{prop}
	\begin{proof}
		We will actually prove that
		$$\theta_{j+1}(q)=\sum_{h=0}^{2^j}u_{2^j-h}(q)u_{2^j+1+h}(q)$$
		and the result will follow by recalling that $u_2(q)=0$. For $j=0$ and $j=1$, by Wu formula (Proposition \ref{Wu}), we have respectively $\theta_1(q)=u_3(q)$ and $\theta_2(q)=u_2(q)u_3(q)+u_5(q)$, which provide our induction basis. Suppose the statement holds for $\theta_j(q)$ with $j \geq 2$, then by Cartan formula and Proposition \ref{Wu} we have that
		\begin{align*}
			&\theta_{j+1}(q)=Sq^{2^j}\theta_j(q)=Sq^{2^j}\sum_{h=0}^{2^{j-1}}u_{2^{j-1}-h}(q)u_{2^{j-1}+1+h}(q)\\
			&=\sum_{h=0}^{2^{j-1}-1}(\tau^{hmod2}u_{2^{j-1}-h}(q)^2Sq^{2^{j-1}+h}u_{2^{j-1}+1+h}(q)\\
			&+\tau^{(h+1)mod2}Sq^{2^{j-1}-h-1}u_{2^{j-1}-h}(q)u_{2^{j-1}+1+h}(q)^2)+Sq^{2^j}u_{2^j+1}(q)\\
			&=\sum_{h=0}^{2^{j-1}-1}(c_{2^{j-1}-h}(q)Sq^{2^{j-1}+h}u_{2^{j-1}+1+h}(q)+Sq^{2^{j-1}-h-1}u_{2^{j-1}-h}(q)c_{2^{j-1}+1+h}(q))\\
			&+\sum_{h=0}^{2^j}u_{2^j-h}(q)u_{2^j+1+h}(q)=\sum_{h=0}^{2^j}u_{2^j-h}(q)u_{2^j+1+h}(q).
		\end{align*} 
	
	\end{proof}
	
	In other words, we obtain that 
	$$u_{2^j+1}(q)=\sum_{h=0}^{2^{j-1}-1}u_{2^{j-1}-h}(q)u_{2^{j-1}+1+h}(q)$$ 
	for any $j$ satisfying $2^j+1 \leq n$.
	
	In \cite{SV}, Smirnov and Vishik highlighted the deep relation between subtle Stiefel-Whitney classes and the $J$-invariant of quadrics defined in \cite{V}. More precisely, they proved the following result.
	
	\begin{thm}\label{J}
		Let $q$ be an $n$-dimensional quadratic form, $p=q$, for even $n$, and $p=q \perp \langle det_{\pm}(q) \rangle$, for odd $n$. Then, 
		$$u_{2j+1}(p) \in (u_{2l+1}(p)|0 \leq l<j) \Rightarrow j \in J(q).$$
	\end{thm}
\begin{proof}
See \cite[Corollary 3.2.22]{SV}. 
\end{proof}
	
	From the previous theorem and from Proposition \ref{formula} we immediately deduce the following well known corollary.
	
	\begin{cor}\label{corJ}
		For any $n$-dimensional $q \in I^3$, $2^{j-1} \in J(q)$ for any $j$ satisfying $2^j+1 \leq n$.
	\end{cor}

\section{The Chern subring of $Ch(B_{\acute{e}t}Spin_n)$}

In this last section we obtain from the structure of $H(BSpin_n)$ some information about the subring generated by Chern classes (coming from the representation given by the map $Spin_n \rightarrow SO_n$) of the Chow ring $Ch(B_{\acute{e}t}Spin_n)$. This is a generalization to more general fields of a result by Yagita (see \cite[Corollary 5.2]{Y}).

First, recall from \cite[Section 1]{EKV} and \cite[Theorem 3.1.1]{SV} that in $H(B_{\acute{e}t}SO_n)$ there are Stiefel-Whitney classes, which we denote by $\widetilde{w_i}$, in bidegree $(i)[i]$ that are mapped to $\tau^{[(i+1)/2]}u_i$ by the homomorphism $H(B_{\acute{e}t}SO_n) \rightarrow H(BSO_n)$.

\begin{lem}\label{van}
	The homomorphism $H(B_{\acute{e}t}SO_n) \rightarrow H(B_{\acute{e}t}Spin_n)$ maps $\widetilde{w_2}$ to $0$.
\end{lem}
\begin{proof}
	It immediately follows from \cite[Theorem 1.14]{EKV}. 
\end{proof}

Note, however, that $c_2$ is not always mapped to $0$ in $H(B_{\acute{e}t}Spin_n)$ as the computations of $Ch(B_{\acute{e}t}Spin_7)$ in \cite{G} and of $Ch(B_{\acute{e}t}Spin_8)$ in \cite{MR} show. This implies that $c_2$ is non zero in $Ch(B_{\acute{e}t}Spin_n)$ for all $n \geq 7$ just by looking at the homomorphisms $Ch(B_{\acute{e}t}Spin_n) \rightarrow Ch(B_{\acute{e}t}Spin_{n-1})$ that send $c_i$ to $c_i$ for all $i \leq n-1$.

In the following lemma we report some formulas holding in $H(BSO_n)$ involving the action of the Milnor operations $Q_i$ on $u_2$. These formulas have formally identical analogues in topology and we present a proof just for completeness.

Before proceeding recall from \cite[Corollary 4]{K} that in our case ($\rho=0$) the Milnor operations can be defined (as in topology) inductively by:
$$Q_0=Sq^1;$$
$$Q_i=Sq^{2^i}Q_{i-1}+Q_{i-1}Sq^{2^i}.$$

\begin{lem}
	In $H(BSO_n)$ for any $i \geq 1$ we have that:\\
	1) $\theta_i=Q_{i-1}u_2$;\\
	2) $Sq^1\theta_{i+1}=\theta_i^2$.
\end{lem}
\begin{proof}
We proceed by induction. For 1), we know that $\theta_1=Sq^1u_2=Q_0u_2$ by definition. Now, suppose $\theta_i=Q_{i-1}u_2$, then
$$Q_iu_2=Sq^{2^i}Q_{i-1}u_2+Q_{i-1}Sq^{2^i}u_2=Sq^{2^i}\theta_i=\theta_{i+1}$$
since for $i=1$ one has that $Q_0Sq^2u_2=Sq^1(u_2^2)=0$ while for $i>1$ the triviality of $Sq^{2^i}u_2$ follows from Wu formula.

For 2), we just need to prove that $Q_iu_3=Sq^1Q_iu_2=\theta_i^2$ for $i \geq 1$. If $i=1$, then 
$$Q_1u_3=Sq^2Sq^1u_3+Sq^1Sq^2u_3=Sq^3u_3=u_3^2=\theta_1^2.$$

Suppose $Q_{i-1}u_3=\theta_{i-1}^2$. Therefore, by Cartan formula
$$Q_iu_3=Sq^{2^i}Q_{i-1}u_3+Q_{i-1}Sq^{2^i}u_3=Sq^{2^i}(\theta_{i-1}^2)=(Sq^{2^{i-1}}\theta_{i-1})^2=\theta_i^2$$
since for $i \geq 2$ we have that $Sq^{2^i}u_3=0$ by Wu formula, which completes the proof.	
\end{proof}

\begin{rem}\label{rem3}
	\normalfont
Note that the element $\tau \theta_i^2$ lives in the Chern subring (see \cite[Section 5]{MRV})
$$Chern(B_{\acute{e}t}SO_n) \cong Chern(BSO_n) \cong \Z/2[c_2,\dots,c_n]$$
of $H(BSO_n)$ for any $i \geq 1$. Moreover, by Lemma \ref{van} and since $\widetilde{w_2}$ maps to $\tau u_2$ via the homomorphism $H(B_{\acute{e}t}SO_n) \rightarrow H(BSO_n) $, we deduce that $\tau \theta_i^2=\tau Sq^1\theta_{i+1}=Sq^1Sq^{2^i} \cdots Sq^1 \widetilde{w_2}$ vanishes in $Chern(B_{\acute{e}t}Spin_n)$ for all $i \geq 1$.
\end{rem}

\begin{prop}
	There exists a ring isomorphism
	$$Chern(B_{\acute{e}t}Spin_n) \cong \Z/2[c_2,\dots,c_n]/{\mathcal I}_n$$
	where the ideal ${\mathcal I}_n$ satisfies the following chain of inclusions
	$$(\tau \theta_1^2,\dots,\tau \theta_{k(n)-1}^2) \subseteq {\mathcal I}_n \subseteq \iota^{-1}(I_{k(n)})$$
	and $\iota: \Z/2[c_2,\dots,c_n] \rightarrow H(BSO_n)$ is the inclusion of the Chern subring of $H(BSO_n)$.
\end{prop}
\begin{proof}
	The ideal ${\mathcal I}_n$ is just the kernel of the epimorphism $\Z/2[c_2,\dots,c_n] \rightarrow Chern(B_{\acute{e}t}Spin_n)$. Then, the first inclusion of the chain is justified by Remark \ref{rem3}.
	
	The second inclusion follows from the fact that the epimorphism $\Z/2[c_2,\dots,c_n] \rightarrow Chern(BSpin_n)$ factors through $Chern(B_{\acute{e}t}Spin_n)$ and by Theorem \ref{MQ2}.
\end{proof}

One can easily see that passing to the radicals in the chain of inclusions above gives
$$\sqrt{(\tau \theta_1^2,\dots,\tau \theta_{k(n)-1}^2)} \subseteq \sqrt{ {\mathcal I}_n} \subseteq \sqrt{(\theta_0^2,\tau \theta_1^2,\dots,\tau \theta_{k(n)-1}^2)}$$
which implies that, modulo nilpotent elements, there exists the following composition of epimorphisms
$$\Z/2[c_2,\dots,c_n]/\sqrt{(\tau \theta_1^2,\dots,\tau \theta_{k(n)-1}^2)} \rightarrow Chern(B_{\acute{e}t}Spin_n)_{red} \rightarrow \Z/2[c_2,\dots,c_n]/\sqrt{(\theta_0^2,\tau \theta_1^2,\dots,\tau \theta_{k(n)-1}^2)}.$$

As we have already mentioned, the previous result is analogous to \cite[Corollary 5.2]{Y} (which is stated over complex numbers) but it is valid more generally without further restriction on the base field (provided that $\rho=0$). This suggests that studying Nisnevich classifying spaces could also be useful for the understanding of the Chow ring of \'etale classifying spaces over more general fields where one usually lacks topological insights. 
	
\footnotesize{

\noindent {\scshape Mathematisches Institut, Ludwig-Maximilians-Universit\"at M\"unchen}\\
fabio.tanania@gmail.com
	
\end{document}